\documentclass[12pt]{article}
\usepackage{amssymb, amsmath,amsthm,mathtools,mathrsfs,color}
\usepackage[all]{xy}
\usepackage{enumerate}
\usepackage{enumitem}

\usepackage{stmaryrd} % mapsfrom
\usepackage{pgffor}
\usepackage{tikz,tikz-cd}
\usepackage{hyperref}
\hypersetup{
    colorlinks,
    linkcolor=blue, %{red!50!black},
    citecolor={blue},
    urlcolor={blue}
}
\usepackage{geometry}
\setlist{itemsep=1pt,topsep=5pt,parsep=0pt} 
\setlist[enumerate]{label={\rm (\arabic{*})}}

\let\C\relax  % 删除hyperref包中对\C的定义

\definecolor{shadecolor}{RGB}{241, 241, 255}

\usetikzlibrary{calc}
\tikzset{
	dot/.style={draw,circle,inner sep=0.8pt,fill=black},
	cir/.style={draw,circle,inner sep=0.8pt}
}
\DeclareMathAlphabet{\mathcal}{OMS}{cmsy}{m}{n}

\theoremstyle{plain}
\newtheorem{Def}{Definition}[section]
\newtheorem{Prop}[Def]{Proposition}
\newtheorem{Theo}[Def]{Theorem}
\newtheorem{Lem}[Def]{Lemma}
\newtheorem{Coro}[Def]{Corollary}

\DeclareMathOperator{\add}{add}

\DeclareMathOperator{\Ext}{Ext}

\DeclareMathOperator{\End}{End}

\DeclareMathOperator*{\hocolim}{Hocolim}

\DeclareMathOperator*{\colim}{colim}

\DeclareMathOperator{\Hom}{Hom}
\DeclareMathOperator{\Img}{Im}

\DeclareMathOperator{\thick}{thick}
\DeclareMathOperator{\Add}{Add}
\DeclareMathOperator{\cone}{cone}

\DeclareMathOperator{\shift}{shift}

\DeclareMathOperator{\Ker}{Ker}
\DeclareMathOperator{\rad}{rad}

\newcommand{\defCategory}[2]{
	\newcommand{#1}{#2\defvariable}}

\NewDocumentCommand{\defvariable}{O{}O{}}{
	\if\relax\detokenize{#2}\relax 
	\if\relax\detokenize{#1}\relax 
	\else 
	({#1})
	\fi
	\else
	\if\relax\detokenize{#1}\relax
	^{{\rm #2}}
	\else 
	^{{\rm #2}}({#1}) 
	\fi 
	\fi } 

\defCategory{\C}{\mathcal{C}}
\defCategory{\K}{\mathcal{K}}
\defCategory{\D}{\mathcal{D}}
\defCategory{\sA}{\mathcal{A}}
\defCategory{\sB}{\mathcal{B}}
\defCategory{\sC}{\mathscr{C}}
\defCategory{\sD}{\mathscr{D}}

\newcommand{\Kb}[1][]{\K[#1][b]}
\newcommand{\Kf}[1][]{\K[#1][-]}
\newcommand{\Kfb}[1][]{\K[#1][-,b]}

\newcommand{\Db}[1][]{\D[#1][b]}
\newcommand{\Df}[1][]{\D[#1][-]}

\newcommand{\Dw}[1][]{\D[#1]}
\newcommand{\Z}{\mathbb{Z}}
\newcommand{\N}{\mathbb{N}}

\newcommand{\Y}{{\sf Y}}

\def\modcat#1{{\sf mod}\mbox{-}{#1}}

\def\Modcat#1{{\sf Mod}\mbox{-}{#1}}

\def\projcat#1{{\sf proj}\mbox{-}{#1}}

\def\injcat#1{{\sf inj}\mbox{-}{#1}}

\def\precompletion#1#2{{\mathbf{L}_{#2}}\left({#1}\right)}
\def\boundfun#1#2{{\mathbf{C}}_{#2}\left({#1}\right)}
\def\completion#1#2{{\mathbf{S}_{#2}}\left({#1}\right)}
\def\precompletionc#1#2{{\mathbf{L}'_{#2}}\left({#1}\right)}
\def\completionc#1#2{{\mathbf{S}'_{#2}}\left({#1}\right)}
\newcommand{\lra}{\longrightarrow}
\newcommand{\lla}{\longleftarrow}
\newcommand{\ra}{\rightarrow}

\newcommand{\lraf}[1]{\stackrel{#1}{\lra}}
\newcommand{\llaf}[1]{\stackrel{#1}{\lla}}
\newcommand{\raf}[1]{\stackrel{#1}{\ra}}

\newcommand{\opp}{{}^{\rm op}}
\newcommand{\cpx}[1]{{#1^{\bullet}}}

\def\A{{\mathscr{A}}}
\def\B{{\mathscr{B}}}

\def\M{{\mathcal{M}}}
\def\N{{\mathcal{N}}}
\def\U{{\mathscr{U}}}
\def\T{{\mathscr{T}}}

\def\G{{\mathcal{G}}}

\def\Abcat{{\bf{Ab}}}
\def\slicesilting{\Omega^{\rm{slice}}}
\def\stacksilting{\Omega^{\rm{stack}}}

\renewcommand{\leq}{\leqslant}
\renewcommand{\geq}{\geqslant}

\title{Metrics on triangulated categories and restrictions of (co)-$t$-structures}
\author{Wei Hu, Ziheng Liu}
\date{}

\begin{document}

\maketitle

\renewcommand{\thefootnote}{\alph{footnote}}
%\setcounter{footnote}{-1} \footnote{ $^*$ Corresponding author.}
%%Email: xicc@cnu.edu.cn; Fax: 0086 10 68903637.}
%\renewcommand{\thefootnote}{\alph{footnote}}
\setcounter{footnote}{-1} 
\renewcommand{\thefootnote}{\alph{footnote}}
\setcounter{footnote}{-1} 

\begin{abstract} 
This paper explores the restriction behavior of silting-induced $t$-structures and co-$t$-structures on triangulated categories endowed with metrics. For compactly generated triangulated categories admitting small coproducts, silting subcategories of compact objects give rise to canonical $t$-structures. We establish that a silting subcategory being contravariantly finite in the precompletion (or completion) is equivalent to the canonical $t$-structure restricting to this precompletion (or completion). This result yields a purely categorical characterization of right coherent rings: a ring $R$ is right coherent if and only if the standard $t$-structure on $\mathcal{D}({\sf Mod}\text{-}R)$ restricts to a $t$-structure on 
$\mathcal{K}^{-,b}({\sf proj}\text{-}R)$. Furthermore, we show that the correspondences between silting objects, bounded (co)-$t$-structures, and simple-minded collections given by Koenig and Yang can be extended to the metric framework of triangulated categories, and still commute with mutation operations and preserve natural partial orders.
\end{abstract}

\section{Introduction}

Triangulated categories lie at the core of modern homological algebra, representation theory of algebras, and algebraic geometry, serving as a unifying framework to study derived categories, homotopy categories, and various cohomological invariants. Among the fundamental tools to dissect the structure of triangulated categories, $t$-structures and co-$t$-structures stand out: $t$-structures, introduced by Beilinson, Bernstein and Deligne \cite{Beilinson1982aa}, allow us to extract abelian hearts from triangulated categories and establish connections between homotopical and abelian categories; co-$t$-structures  provide a dual perspective.

In recent years, a systematic theory of metrics on triangulated categories has been developed (\cite{Neeman2018TheC,Neeman2026aa,Krause2020aa,Krause2024}). Initially, this theory was designed to provide an alternative proof for a key result in Rickard’s Morita theory: a triangle equivalence between bounded derived categories of finitely presented modules over two right coherent rings restricts to a triangle equivalence between homotopy categories of bounded complexes of finitely generated projective modules. Beyond this initial application, the theory quickly emerged as a powerful and flexible language to address a wide range of structural problems on triangulated categories. The theory is compatible with $t$-structures and co-$t$-structures: every $t$-structure or co-$t$-structure on a triangulated category naturally induces a good metric, and conversely, metrics can be used to characterize the existence of these core structures (see \cite{Neeman2024,Biswas2024}).

In a compactly generated triangulated category with small coproducts, a set of compact objects always induces a $t$-structure \cite{Tarrio2003}*{Theorem A.1}, see also \cite{Neeman2021tstr}. Particularly, a silting subcategory of compact objects  generates canonical $t$-structures that generalize the standard $t$-structures on derived categories of rings and dg algebras. For a compactly generated triangulated category $\T$ with small coproducts, let $\T^c$ denote the full subcategory of compact objects; if $\T^c$ admits a silting subcategory $\G$ of compact objects, $\G$ canonically induces a $t$-structure $(\mathscr{T}_{\mathcal{G}}^{\leq 0},\mathscr{T}_{\mathcal{G}}^{\geq 0})$ on $\T$ via Hom-orthogonality. This silting-induced $t$-structure further defines a good metric on $\T^c$, whose precompletion $\T_c^-$ and completion $\T_c^b$ recover classical categories in ring theory: for a ring $R$, $\mathscr{T}_c^-=\Kb[\projcat{R}]$ and $\T_c^b=\Kfb[\projcat{R}]$. When $R$ is right coherent, $\Kfb[\projcat{R}]$ is triangle equivalent to the bounded derived category $\Db[\modcat{R}]$ of finitely presented modules, and the standard $t$-structure on $\D[\Modcat{R}]$ is known to restrict to this subcategory.
This classical observation raises a natural question in the general setting: under what conditions does the silting-induced $t$-structure  restrict to a $t$-structure on the precompletion  or completion? Existing literature has focused on lifting $t$-structures to larger triangulated categories, but the restriction problem to metric completions of compact subcategories remains largely uncharacterized, especially for general silting subcategories beyond the standard projective generator case. Moreover, the interplay between restricted $t$-structures, co-$t$-structures, and categorical finiteness conditions (such as contravariant finiteness) has not been systematically explored.

\medskip 
In this paper, we establish a complete characterization of the restriction behavior of silting-induced $t$-structures, and further extend the classical Koenig-Yang correspondences to the metric framework of triangulated categories. Our main results can be  summarized as follows.

\begin{Theo}\label{theo-main-intro}
Let $\T$ be a compactly generated triangulated category with small coproducts, and let $\G$ be a silting subcategory of a $\T^c$. Then, for the following conditions
\begin{enumerate}[label={\rm (\roman*)}]
\item $\G$ is contravariantly finite in $\T^c$;
\item $\G$ is contravariantly finite in $\T_c^-$;
\item The $t$-structure $(\T_{\G}^{\leq 0}, \T_{\G}^{\geq 0})$ induced by $\G$ restricts to $\T_c^-$;
\item The $t$-structure $(\T_{\G}^{\leq 0}, \T_{\G}^{\geq 0})$ induced by $\G$ restricts to $\T_c^b$;
\item Every object in $\T_c^b$ admits a right $\add(\G)$-approximation;
\end{enumerate}
one has ${\rm (i)}\Leftrightarrow {\rm (ii)}\Leftrightarrow {\rm (iii)}\Rightarrow {\rm (iv)}\Rightarrow {\rm (v)}$. If $\T^c\subseteq\T_c^b$, then {\rm (v)} also implies {\rm (i)}.
Moreover, if the equivalent conditions {\rm (i)-(iii)} are satisfied, then the following hold. 
\begin{enumerate}
	\item The $t$-structure $(\T_{\G}^{\leq 0}, \T_{\G}^{\geq 0})$ restricts to an above-bounded $t$-structure on  $\T_c^-$ and to a bounded $t$-structure on $\T_c^b$. The two restricted $t$-structures have the same heart $\mathcal{H}_{\G}$, which is equivalent to $\modcat{\G}$ and $\modcat{H^0(\G)}$ in the following way. 
	 \[\arraycolsep=3pt
    \begin{array}{rcccl}
    \modcat{\G} &\llaf{\simeq} &\mathcal{H}_{\G}&\lraf{\simeq}& \modcat{H^0(\G)}\\
    \Hom(-,X)|_{\G} &\mapsfrom & X &\mapsto & \Hom(-,X)|_{H^0(\G)},
    \end{array}\] 
	\item The intersection $\T_c^-\cap \T_{\G}^{\leq 0}$ determine a co-$t$-structure  $(\langle\G\rangle^{[0,+\infty)},\T_c^{-}\cap\T_{\G}^{\leq 0})$  on $\T_c^-$ with co-heart $\G$.
	If $\T^c\subseteq \T_c^b$, then the co-$t$-structure on $\T_c^-$ can be restricted further to $\T_c^b$ with co-heart $\G$. 
\end{enumerate}
\end{Theo}

An immediate consequence   is a purely categorical characterization of right coherent rings. 

\begin{Coro}
Let $R$ be a ring. Then the following are equivalent.
\begin{enumerate}
    \item $R$ is right coherent.
    \item The standard $t$-structure of $\D[\Modcat{R}]$ restricts to $\Kf[\projcat{R}]$. 
    \item The standard $t$-structure of $\D[\Modcat{R}]$ restricts to $\Kfb[\projcat{R}]$. 
\end{enumerate}
\end{Coro}

This corollary generalizes a recent result appearing in \cite{Marks2023}*{Corollary 4.3} and  \cite{Saorin2022}*{Proposition 6.6} where they consider the restriction of the $t$-structure induced by a tilting complex $\cpx{T}$ over a right coherent ring $R$ to the bounded derived category $\Db[\modcat{R}]$, and say that the $t$-structure restricts to $\Db[\modcat{R}]$ if and only if the endomorphism algebra of $\cpx{T}$ is right coherent. It is known that 
 right coherent rings are not closed under derived equivalences. This means that the restriction of equivalent $t$-structures to a smaller subcategory may differ.

\medskip 
Based on the above results, we can extend a classic result of Koenig and Yang \cite{Koenig2014aa} to the metric framework. 

\begin{Theo}\label{theo-correspondence}
Let $k$ be a field, and let $\T$ be a compactly generated triangulated category with small coproducts such that $\T^c$ admits a silting object $G$ with $\Hom(G,G[n])$ being finite dimensional for all $n$. Suppose that there is an object $G'\in \T_c^b$ and an integer $N$ such that $\T=\overline{\langle G'\rangle}_N$. Then there are bijections which commute
with mutations and which preserve partial orders between
\begin{enumerate}
	\item ismorphism class of basic silting objects in $\T^c$;
	\item bounded co-$t$-structures on $\T^c$;
	\item bounded $t$-structures on $\T_c^b$ with length heart;
	\item equivalence classes of simple-minded collections in $\T_c^b$.  
\end{enumerate}
\end{Theo}

This paper is organized as follows. Section 2 collects preliminary definitions and results on triangulated categories, good metrics, completions, silting subcategories, $\mathsf{t}$-structures and co-$\mathsf{t}$-structures, fixing notations and recalling key lemmas from the literature. Section 3 is devoted to the proof of Theorem \ref{theo-main-intro}, analyzing the restriction of $\mathsf{t}$-structures and co-$\mathsf{t}$-structures induced by silting subcategories via approximation theory and metric completions. Section 4 develope a generating theorem and Section 5 establishes the extended Koenig-Yang correspondences (Theorem \ref{theo-correspondence}), verifying that the bijections commute with mutations and preserve partial orders.

\section{Preliminaries}

\subsection{General convention}
For an additive category $\mathscr{C}$ and a full subcategory $\G$ of $\mathscr{C}$, we denote by $\add(\G)$ the additive closure of $\G$ in $\mathscr{C}$. That is, $\add(\G)$ consists of objects isomorphic to direct summands of finite direct sums of ojects in $\G$. If $\mathscr{C}$ admits small coproducts, we write $\Add(\G)$ for the full subcategory of $\mathscr{C}$ consists of objects isomorphic to direct summands of coproducts of objects in $\G$. The coproduct of a family of objects $X_i$ in $\mathscr{C}$ is denoted by $\coprod X_i$. An object $G$ in $\mathscr{C}$ is called {\em compact} if for any family of objects $X_i, i\in I$ in $\mathscr{C}$, the canonical morphism 
\[\coprod_{i\in I} \Hom_{\mathscr{C}}(G,X_i)\lra \Hom_{\mathscr{C}}(G,\coprod_{i\in I}X_i)\]
is an isomorphism. The full subcategory consisting of all compact objects in $\mathscr{C}$ is denoted by $\mathscr{C}^c$. For an additive category $\mathscr{C}$ and a full subcategory $\G$ of $\mathscr{C}$, we denote by $\G^{\perp}$ and ${}^{\perp}\G$ the full subcategories of Hom-orthogonal objects with respect to $\G$, that is,
$$\G^{\perp}=\left\{X\in\mathscr{C}:\Hom_{\mathscr{C}}(G,X)=0\text{ for all }G\in\G\right\},$$
$${}^{\perp}\G=\left\{X\in\mathscr{C}:\Hom_{\mathscr{C}}(X,G)=0\text{ for all }G\in\G\right\}.$$

Let $\T$ be a triangulated category. The shift fuctor of $\T$ is denoted by $[1]$. Throughout this paper, if no confusion arises, we will write 
$\Hom(-,-)$ for $\Hom_{\T}(-,-)$ and $\cone(f)$, called the cone of $f$, for the third term of a triangle $X\raf{f} Y\ra \cone(f)\ra X[1]$. The extension of  two full subcategories $\A$ and $\B$ is the full subcategory
\[\A*\B:=\left\{X\in\T\mid \text{there is a triangle }A\to X\to B\to A[1]\text{ such that }A\in\A,\,B\in\B\right\}.\]
Given three full subcategories $\mathscr{A},\mathscr{B}$ and $\mathscr{C}$ of $\T$, The Octahedral Axiom implies the associativity  $(\mathscr{A}*\mathscr{B})*\mathscr{C}=\mathscr{A}*(\mathscr{B}*\mathscr{C})$. Given a sequence of morphisms $X_*: X_1\lraf{f_1}X_2\lraf{f_2}\cdots$ in $\T$ such that the coproduct $\coprod X_i$ exists in $\T$, the homotopy colimit $\hocolim X_i$ is defined up to isomorphism by the triangle
\[\xymatrix{
\coprod X_i\ar[r]^{1-\shift} & \coprod X_i\ar[r] & \hocolim X_i\ar[r] & \coprod X_i[1].
}\]
If there is an integer $N$ such that $f_i$ is an isomorphism for all $i\geq N$, then $\hocolim X_i\simeq X_N$. If $G$ is a compact object of $\T$, then applying $\Hom(G,-)$ to the triangle above implies that 
\[\Hom(G,\hocolim X_i)\simeq \colim \Hom(G,X_i).\] 
A triangulated category $\T$ is said to be {\em compactly generated} by $\G\subseteq\T^c$ if, for each non-zero object $Y\in\T$, there is some $G\in\G$ and $n\in\mathbb{Z}$ such that $\Hom(G,Y[n])\neq 0$.

\begin{Lem}
\label{lem-desc-hocolim}
Suppose that $\T$ is a triangulated category with small coproducts, compactly generated by $\G$. If $X_1\lraf{f_1}X_2\lraf{f_2}\cdots$ is a sequence in $\T$ such that for each integer $k$, $X_i\in (\G[m])^{\perp}$ when $i$ is sufficiently large, then $\hocolim X_i\simeq 0$.
\end{Lem}
\begin{proof}
For $G\in\G$ and integer $m$, the assumption just means that $\Hom(G[m], X_i)$ is zero when $i$ is sufficiently large. This implies that 
\[\Hom(G[m],\hocolim X_i)\simeq \colim \Hom(G[m],X_i)=0.\] 
Since $\G$ generates $\T$, we have $\{\G[m]\mid m\in\mathbb{Z}\}^{\perp}=\{0\}$. It follows that $\hocolim X_i\simeq 0$.
\end{proof}

\subsection{Metrics on triangulated categories and completions}
In this subsection, we recall definitions and basic facts of metrics on triangulated categories. Most are from Neeman's paper \cite{Neeman2018TheC}. 
\begin{Def}[\cite{Neeman2018TheC}]
\label{def-metric}
	Suppose $\T$ is a triangulated category. A {\em good metric} on $\T$ is a sequence of additive full subcategories of $\T$ $$\M=\{\M_i\subset\T\mid \,i\in\Z\}$$
	such that for all integers $i$, the following two conditions hold:
	\begin{itemize}
		\item [$(1)$] $\M_i*\M_i=\M_i$;
		\item [$(2)$]$\M_{i+1}[-1]\cup\M_{i+1}\cup\M_{i+1}[1]\subset\M_{i}$.
	\end{itemize}
\end{Def}

Two good metrics $\M$ and $\N$ on a triangulated category $\T$ are said to be equivalent if for each integer $n$, there are integers $m_1,m_2$ such that $\N_{m_1}\subseteq \M_n\subseteq N_{m_2}$. 
With respect to a good metric $\M$ on $\T$, a {\em Cauchy sequence} $X_*$ is a sequence of morphisms:
\[X_1\lraf{\alpha_1}X_2\lraf{\alpha_2}\cdots\]
in $\T$ such that for each integer $i$, there is a positive integer $N$ such that the cone of the composition $\alpha_n\alpha_{n-1}\cdots \alpha_m$ belongs to $\M_i$ for all $n>m>N$, or equivalently, for each $i\in\mathbb{Z}$, there is a positive integer $N$ such that  $\cone(\alpha_n)$  belongs to $\M_i$ for all $n>N$ (see \cite[Lemma 1.9]{Neeman2018TheC}).  By definition, it is easy to see that equivalent good metrics on $\T$ have the same Cauchy sequences.

%\begin{Def}\cite{Neeman2018TheC} Let $\left\{\M_i:i\in\Z\right\}$be a metric on the triangulated category $\T$, a sequence of morphisms $E_1\lraf{f_1}E_2\lraf{f_2}E_3\lraf{f_3}\cdots$ is called a Cauchy sequence with respect to the metric $\left\{\M_i:i\in\Z\right\}$ if for all $i\in\Z$, there exists $N>0$ such that for all$n>m>N$, the triangle $$E_m\lra E_n\lra D_{m,n}\lra E_m[1]$$ satisfies that $D_{m,n}\in\M_i$, where $E_m\to E_n$ is the composition of morphisms $$E_m\lraf{f_m}E_{m+1}\lraf{f_{m+1}}\cdots\lraf{f_{n-1}}E_n.$$ \end{Def}

%\begin{Prop}\label{easy-Cauchy-seq}\cite[Section 2.3]{ChenHongxing2024aa}Let $\left\{\M_i:i\in\Z\right\}$ be a metric on the triangulated category $\T$, the sequence of morphisms  $E_1\lraf{f_1}E_2\lraf{f_2}E_3\lraf{f_3}\cdots$ is a Cauchy sequence if for all $i\in\Z$, there exists $N>0$ such that for all $n>N$, $\cone(f_n)\in\M_i$.\end{Prop}

\medskip
To define the completion of a triangulated category $\T$ with respect to a good metric $\M$, one needs the category $\Modcat{\T}$ consisting of contravariant additive functors from $\T$ to the category $\Abcat$ of abelian groups. The shift functor $[1]$ of $\T$ naturally induces an automorphism on $\Modcat{\T}$, also denoted by $[1]$, which is defined as follows:
\begin{itemize}
	\item[$(1)$] For each $F\in\Modcat{\T}$, let $F[1]=F(-[-1])$, that is,  $F[1](X)=F(X[-1])$ for all objects  $X\in\T$.
	\item[$(2)$] For a morphism $\alpha:F_1\to F_1$ in $\Modcat{\T}$, which is a natural transformation from $F_1$ to $F_2$, the component of $\alpha[1]$ at an object $X\in\T$ is defined as $\alpha_{X[-1]}:F_1(X[-1])\to F_2(X[-1])$.
\end{itemize}
With the help of $\Modcat{\T}$, one can define the completion of  $\T$ with respect to a good metric on $\T$. 
\begin{Def}\cite{Neeman2018TheC}
Let $\M=\{\M_i\mid i\in\Z\}$ be a good metric on a triangulated category $\T$. There are three full subcategories of $\Modcat{\T}$: 	\begin{itemize}
\item [$(1)$]  $\precompletion{\T}{\M}$ consists of those functors $F\in\Modcat{\T}$ isomorphic to $\colim \Hom(-,X_i)$ for some Cauchy sequence $X_*$ in $\T$ with respect to $\M$. 
\item [$(2)$] $\boundfun{\T}{\M}$ consists of those $F\in \Modcat{\T}$ such that $F(\M_n)=0$ for some integer $n$. 
\item [$(3)$] $\completion{\T}{\M}=\precompletion{\T}{\M}\cap\boundfun{\T}{\M}$.
	\end{itemize}
With respect to the metric $\M$, the category $\precompletion{\T}{\M}$ is called the {\em precompletion} of $\T$ , and $\completion{\T}{\M}$ is called the {\em completion} of $\T$.  
\end{Def}
As we expected, the completion $\completion{\T}{\M}$ of a triangulated category $\T$ with respect to a good metric $\M$ is again a triangulated category (see \cite{Neeman2018TheC}*{Theorem 2.11}). A triangle in the completion $\completion{\T}{\M}$ is of the form 
\[\colim\Hom(-,X_i)\lraf{f}\colim\Hom(-,Y_i)\lraf{g}\colim\Hom(-,Z_i)\lraf{h}\colim\Hom(-,X_i)[1],\]
where $X_*, Y_*$, and $Z_*$ are Cauchy sequences with respect to $\M$ fitting into the following commutative diagram
	\[\xymatrix{		X_1\ar[r]\ar[d]^{f_1}&X_2\ar[r]\ar[d]^{f_2}&X_3\ar[r]\ar[d]^{f_3}&\cdots\\
		Y_1\ar[r]\ar[d]^{g_1}&Y_2\ar[r]\ar[d]^{g_2}&Y_3\ar[r]\ar[d]^{g_3}&\cdots\\
		Z_1\ar[r]\ar[d]^{h_1}&Z_2\ar[r]\ar[d]^{h_2}&Z_3\ar[r]\ar[d]^{h_3}&\cdots\\
		X_1[1]\ar[r]&X_2[1]\ar[r]&X_3[1]\ar[r]&\cdots\\
	}\]
such that all columns are triangle in $\T$, and $f,g$ and $h$ is given by the colimit of $f_i,g_i$ and $h_i$, respectively.

\medskip
The metric completion  $\completion{\T}{}$ of $\T$ is a full subcategory of the functor category $\Modcat{\T}$, which seems hard to describe. There is another way to realize $\completion{\T}{}$ as a triangulated subcategory of a triangulated category, by means of ``good extension". Suppose that there is a fully faithful triangle functor
\[F: \T\lra \tilde{\T}\]
from $\T$ to another triangulated category $\tilde{\T}$ which admits small coproducts. Let $\Y_F:\tilde{\T}\lra\Modcat{\T}$ be the functor sending an object $X\in\tilde{\T}$ to $\Hom(F(-),X)$ and a morphism $f$ in $\tilde{\T}$ to $\Hom(F(-),f)$. Then the composition $\Y_F\circ F$ is naturally isomorphic to the Yoneda functor $\Y: \T\lra \Modcat{\T}$. Now suppose that $\T$ admits a good metric $\M$. The functor $F$ is called a {\em good extension} with respect to the metric $\M$ if, for any Cauchy sequence $X_*$ in $\T$ with respect to $\M$, there is an isomorphism
\[\colim F(X_i)\simeq \Y_F(\hocolim F(X_i))\]
functorial in the Cauchy sequence. Note that for each triangulated category $\T$ with small coproducts and for any good metric on the full subcategory $\T^c$ consisting of all compact objects, the embedding functor  $\T^c\to\T$ is always a good extension (\cite[Example 3.9]{Neeman2018TheC}).

Suppose that $\T$ is a triangulated category admits a good metric $\M$, and that $F:\T\lra \tilde{\T}$ is a good extension. Let $\precompletionc{\T}{\M}$ be the full subcategory of $\tilde{\T}$ consisting of objects isomorphic to $\hocolim F(X_i)$ for some Cauchy sequene $X_*$ in $\T$ with respect to $\M$, and let \[\completionc{\T}{\M}:=\precompletionc{\T}{\M}\cap \Y_F^{-1}(\boundfun{\T}{\M}).\]
Neeman proved that $Y_F$ restricts to an equivalence between 
$\precompletionc{\T}{\M}$ and $\precompletion{\T}{\M}$, and a triangle equivalence between $\completionc{\T}{\M}$ and $\completion{\T}{\M}$.  

Let us give some examples of good metrics on triangulated categories. Let $\A$ be a full subcategory of a triangulated category $\T$. For a non-empty subset $\Phi$ of $\mathbb{Z}$, let
\[\A[\Phi]=\left\{A[i]\mid i\in \Phi,\,A\in\A\right\},\]
and let 
\[\langle \A\rangle_1^{\Phi}:=\add(\A[-\Phi]), \mbox{ and }\langle\A\rangle_n^{\Phi}:=\add(\langle \A\rangle_1^{\Phi}*\langle \A\rangle_{n-1}^{\Phi}).\] 
for $n>1$. In case that $\T$ admits small coproducts,  we define 
\[\overline{\langle \A\rangle}_1^{\Phi}:=\Add(\A[-\Phi]), \mbox{ and }\overline{\langle\A\rangle}_n^{\Phi}:=\Add(\overline{\langle \A\rangle}_1^{\Phi}*\overline{\langle \A\rangle}_{n-1}^{\Phi}).\] 
for $n>1$.
Taking union yields
$$\left\langle \A\right\rangle^{\Phi}=\bigcup_{n>0}\left\langle \A\right\rangle_n^{\Phi},\quad\quad\overline{\left\langle \A\right\rangle}^{\Phi}=\bigcup_{n>0}\overline{\left\langle \A\right\rangle}_n^{\Phi}.$$
Defining 
\[\M:=\{\M_n:=\langle\A\rangle^{(-\infty,-n]}\mid n\in\mathbb{Z}\}\]
gives a good metric on $\T$. In case that $\T$ admits small coproducts, 
\[\N:=\{\N_n:=\overline{\langle\A\rangle}^{(-\infty,-n]}\mid n\in\mathbb{Z}\}\]
is a good metric on $\T$. 

For an arbitrary ring $R$, its derived category $\D[\Modcat{R}]$ of arbitrary right $R$-modules is compactly generated. The compact objects are precisely those complexes isomorphic to bounded complexes of finitely generated projective $R$-modules. Thus the embedding $\Kb[\projcat{R}]\hookrightarrow \D[\Modcat{R}]$ is a good extension for any good metric on $\Kb[\projcat{R}]$. Consider the canonical metric $\M:=\left\{\langle R\rangle^{(-\infty,-n]}\mid n\in\mathbb{Z}\right\}$ on $\Kb[\projcat{R}]$, then we have 
\[\precompletionc{\Kb[\projcat{R}]}{\M}=\Kf[\projcat{R}],\quad \completionc{\Kb[\projcat{R}]}{\M}=\Kfb[\projcat{R}].\]
In case that $R$ is a right coherent, let $\modcat{R}$ be the category of finitely presented right $R$-modules. Then $\modcat{R}$ is an abelian category and  
$\N:=\left\{\overline{\left\langle R \right\rangle}^{(-\infty,-n]}\cap\Db[\modcat{R}]\mid n\in\Z\right\}$ is a good metric on $[\Db[\modcat{R}]]\opp$. It is shown \cite[Proposition 5.6]{Neeman2018TheC} that the embedding $[\Db[\modcat{A}]]\opp\hookrightarrow [\Dw[\Modcat{A}]]\opp$ is a good extension with respect to $\N$ and the completions are as follows. 
$$\precompletionc{[\Db[\modcat{R}]]\opp}{\N}=[\Df[\modcat{R}]]\opp,\quad \completionc{[\Db[\modcat{R}]]\opp}{\N}=[\Kb[\projcat{R}]]\opp.$$
%This was used by Neeman  to give an elegant  proof of the fact: two right coherent rings are derived equivalent if and only if their derived categories of bounded complexes of finitely presented modules are triangle equivalent. 

\subsection{Silting subcategories, $t$-structures and co-$t$-structures}
Both $t$-structures and co-$t$-structures on a triangulated category naturally give rise to good metrics on the triangulated category. 

\begin{Def}[\cite{Beilinson1982aa}]
	\label{def-tstr}
	A $t$-structure $t=\left(t^{\leq0},t^{\geq0}\right)$ on a triangulated category $\T$ is a pair of full subcategories satisfying the following conditions.
	\begin{itemize}
		\item [$(1)$] $t^{\leq0}[1]\subset t^{\leq0}$, and $t^{\geq0}[-1]\subset t^{\geq0}$,
		\item [$(2)$] $\Hom\left(t^{\leq0},t^{\geq0}[-1]\right)=0$,
		\item[$(3)$] for any $X\in\T$, there is a triangle $$\tau^{\leq 0}X\to X\to \tau^{\geq 1}X\to (\tau^{\leq 0}X)[1]$$ such that $\tau^{\leq 0}X\in t^{\leq0}$, $\tau^{\geq 1}X\in t^{\geq0}[-1]$. 
	\end{itemize}
\end{Def}

For a $t$-structure $t=\left(t^{\leq0},t^{\geq0}\right)$ on a triangulated category $\T$, we define 
\[t^{\leq n}=t^{\leq0}[-n],\quad t^{\geq n}=t^{\geq0}[-n] \mbox{ and } t^n=t^{\leq n}\cap t^{\geq n}.\]
It is well-known that $t^0$ is an  abelian category, called the {\em heart} of $t$. In particular, for any $X,\,Y\in t^0$ and integer $m<0$, we have $\Hom\left(X,Y[m]\right)=0$. For each integer $n$, the triangle in Definition \ref{def-tstr} gives rise to functors $\tau^{\leq n}=[-n]\circ \tau^{\leq 0}(-[n]):\T\ra t^{\leq n}$ and $\tau^{\geq n+1}=[-n]\circ\tau^{\geq 1}(-[n]):\T\ra t^{\geq n+1}$. The $n$th homological functor $H^n:\T\ra t^0$ is defined to be $H^n=[n]\circ\tau^{\geq n}\circ\tau^{\leq n}$, which is also isomorphic to $[n]\circ\tau^{\leq n}\circ\tau^{\geq n}$.

Two $t$-structures   $t_1=\left(t_1^{\leq0},t_1^{\geq0}\right)$ and $t_2=\left(t_2^{\leq0},t_2^{\geq0}\right)$  on a triangulated category $\T$ are called {\em equivalent} if there exists an integer $n\geq0$ such that $t_1^{\leq -n}\subset t_2^{\leq0}\subset t_1^{\leq n}$.  For $t$-structure $t=\left(t^{\leq0},t^{\geq0}\right)$ on a triangulated category $\T$, one can define the following full subcategories of $\T$. 
\[\T^-\coloneq\bigcup_{n\in\Z}t^{\leq n}, \quad \T^+\coloneq\bigcup_{n\in\Z}t^{\geq n}, \quad \T^b\coloneq\T^-\cap\T^+.\]
The $t$-structure $t$ is called {\em bounded above} ({\em bound below, bounded}, respectively) provided that $\T^-=\T$ ($\T^+=\T, \T^b=\T$, respectively), and is called {\em non-degenerate} if $\bigcap_{n\in\Z}t^{\leq n}=0=\bigcap_{n\in\Z}t^{\geq n}$. By definition, equivalent $t$-structures induce the same subcategories $\T^-$, $\T^+$ and $\T^b$.

\begin{Def}[\cite{Bondarko2007aa,Pauks2008aa}]
\label{def-cot}
	A co-$t$-structure $c=\left(c_{\geq0},c_{\leq0}\right)$ on a triangulated category $\T$ is a pair of full  subcategories satisfying the following conditions.
	\begin{itemize}
		\item [$(1)$] $c_{\geq0}[-1]\subset c_{\geq0}$  and $c_{\leq0}[1]\subset c_{\leq0}$.
		\item [$(2)$] $\Hom\left(c_{\geq0},c_{\leq0}[1]\right)=0$.
		\item [$(3)$] For any $X\in\T$, there is a triangle 
		$$A\to X\to B\to A[1]$$
		such that $A\in c_{\geq0}$ and $B\in c_{\leq0}[1]$.
		\item [$(4)$] The subcategories $c_{\geq0}$ and $c_{\leq0}$ are closed under direct summands.
	\end{itemize}
\end{Def}
Given a co-$t$-structure $c=\left(c_{\geq0},c_{\leq0}\right)$ on a triangulated category $\T$,  for any $n\in\Z$, we set $c_{\geq n}=c_{\geq0}[-n]$, $c_{\leq n}=c_{\leq0}[-n]$. The intersection $c_0=c_{\geq0}\cap c_{\leq0}$ is called the {\em co-heart} of the co-$t$-structure $(c_{\geq0},c_{\leq0})$. Similarly one can define the concepts of bounded-above, bounded-below and bounded co-$t$-structures. Recall that a full additive subcategory $\G$ of $\T$ is called a {\em presilting subcategory} if $\Hom(G,G'[n])=0$ for all $G,G'\in\G$ and $n>0$, and is called a {\em silting subcategory} if $\G$ is presilting and $\thick(\G)=\T$.  If a silting subcategory $\G=\add(G)$ for some object $G\in\T$, then $G$ is called a {\em silting object} in $\T$. 

The following lemma summaries some basic facts on $t$-structures and co-$t$-structures. 
\begin{Lem}
\label{lem-facts-t-cot}
Let $\T$ be a triangulated category. Then the following hold.
\begin{enumerate}
    %\item  If $t$ is a $t$-structure on $\T$, then $t^{\leq n}={}^{\perp}(t^{\geq n+1})$ and $t^{\geq n+1}=(t^{\leq n})^{\perp}$ for all integers $n\in \mathbb{Z}$. In particular, $t^{\leq n}$ and $t^{\geq n}$ are closed under extension for all $n\in\mathbb{Z}$. 
   
    %\item[$(3)$] If $c$ is a co-$t$-structure on $\T$, then both  $c_{\leq n}$ and $c_{\geq n}$ are closed under extension for all $n\in\mathbb{Z}$.
    \item The co-heart of a bounded co-$t$-structure of $\T$ is a silting subcategory of $\T$.
    \item If $\G$ is  a silting subcategory of $\T$,  then 
	\[\left(\left\langle \G\right\rangle^{[0,\infty)}, \left\langle \G\right\rangle^{(-\infty,0]}\right)\]
is a bounded co-$t$-structure on $\T$ with co-heart $\G$. 
 \item Two $t$-structures $t_1$ and $t_2$ on $\T$ coincide if and only if $t_1^{\leq0}\subset t_2^{\leq0},\,t_1^{\geq0}\subset t_2^{\geq0}$.
\item Two co-$t$-structures $c^1$ and $c^2$ coincide if and only if $c^1_{\leq0}\subset c^2_{\leq0}$ and $c^1_{\geq0}\subset c^2_{\geq0}$. 
    \item[$(7)$] If $t$ is a $t$-structure on $\T$, then $\left\{t^{\leq -n}\mid n\in\Z\right\}$ and $\left\{t^{\geq n}\mid n\in\Z\right\}$ are good metrics on $\T$.

\item  If $c$ is a co-$t$-structure on $\T$, then both $\left\{c_{\leq-n}\mid n\in\Z\right\}$ and $
		\left\{c_{\geq n}\mid n\in\Z\right\}$ are good metrics on $\T$.
\end{enumerate}
\end{Lem}

\iffalse 
A silting subcategory of a triangulated category always induces a co-$t$-structure.
\begin{Prop}\label{silting-co-t-struc}(\cite[Section 2]{Aihara2012aa})
	Let $\G$ be a silting subcategory of a triangulated category $\T$.  Then 
	\[\left(\left\langle \G\right\rangle^{[0,\infty)}, \left\langle \G\right\rangle^{(-\infty,0]}\right)\]
is a bounded co-$t$-structure on $\T$ with co-heart $\G$. 
\end{Prop}
\fi

\section{Restrictions of $t$-structures and co-$t$-structures induced by silting subcategories}
\label{sect-restr}
In this section, we fix a compactly generated triangulated category $\T$ with small coproduct. Let $\T^c$ be the full subcategory consisting of compact objects. It was proved in \cite[Theorem A.1]{Tarrio2003} that any set $\G$ of compact objects in $\T$ would generate a $t$-structure $(\overline{\langle\G\rangle}^{(-\infty,0]}, (\G[\mathbb{Z}_+])^{\perp})$ on $\T$. We denote this $t$-structure by $(\T_{\G}^{\leq 0},\T_{\G}^{\geq 0})$.  In case that the additive closure $\G$ of the generating set is a silting subcategory of $\T^c$, this was also proved  in \cite[Chapter 3 Theorem 2.3, Proposition 2.8]{Beligiannis2007aa} (see also \cite[Theorem1.3]{Hoshino2002aa}). In this case, the left part $\T_{\G}^{\leq 0}=(\G[\mathbb{Z}_-])^{\perp}$. 

We only consider the silting case and assume that $\G$ is generating set of compact objects such that $\add(\G)$ is a silting subcategory of $\T^c$. In this case, 
\[\left(\left\langle \G\right\rangle^{[0,\infty)}, \left\langle \G\right\rangle^{(-\infty,0]}\right)\] is a bounded co-$t$-structure on $\T^c$ with co-heart $\G$. We have the following proposition.
\begin{Prop}
\label{prop-silt-completion}
Suppose that $\T$ is a compactly generated triangulated category with small coproducts and that $\G$ is a silting subcategory of $\T^c$.  Let $\M^{\G}=\{\langle \G\rangle^{(-\infty,-n]}\mid n\in\mathbb{Z}\}$ be the good metric on $\T^c$ induced by $\G$. Then 
\[\precompletionc{\T^c}{\M^{\G}}=\T_c^{-}:=\bigcap_{n\in\mathbb{Z}}(\T^c*\T_{\G}^{\leq n}),\quad \completionc{\T^c}{\M^{\G}}=\T_c^b:=\T_c^-\cap \left(\bigcup_{n\in\mathbb{Z}}\T_{\G}^{\geq n}\right)\]
\end{Prop}

In case that $\G=\add(G)$ for some object $G$ of $\T^c$,  Proposition \ref{prop-silt-completion} follows from \cite[Theorem 3.15, Example 4.2, Proposition 5.6]{Neeman2018TheC} in the sense of (weakly) approximable triangulated categories. In the general case that $\G$ is a silting subcategory, one may refer to \cite{Rahul2025} for proof for ``(weakly) $\G$-approximable triangulated categories". For the convenience of the reader, we give a direct proof here.

\begin{proof}
    We frist claim that $\precompletionc{\T^c}{\M^{\G}}\subseteq \T_c^-$.  Take $X\in\precompletionc{\T^c}{\M^{\G}}$,  there is a Cauchy sequence $E_*: E_1\lraf{f_1}E_2\lraf{f_2}\cdots$
	in $\T^c$ such that $X=\hocolim E_i$. For each integer $d$, since $E_*$ is a Cauchy sequence, there exists $N>0$ such that for any $n\geq N$, the cone  of $f_n$ belongs to $\left\langle \G\right\rangle^{(-\infty,d-1]}$. Thus $\cone(f_n)[-1]\in \left\langle \G\right\rangle^{(-\infty,d]}$. For each $Q\in\langle\G\rangle^{[d+1,+\infty]}$, applying $\Hom(Q, -)$ to the triangle 
	\[\cone(f_n)[-1]\lra E_n\lraf{f_n}E_{n+1}\lra \cone(f_n)\]
implies that $\Hom(Q, f_n)$ is an isomorphism in this case. Consider the triangle 
\[\xi_N:\quad E_N\lraf{\lambda}\hocolim E_i\lra Y\lra E_N[1],\] where $\lambda$ is the canonical morphism. For each $G\in\G$ and $m<-d$, we have $G[m]\in \langle\G\rangle^{[d+1,+\infty]}$ and that $\Hom(G[m],f_n)$ is an isomorphism for all $n\geq N$. It follows that 
\[\Hom(G[m],\lambda): \Hom(G[m],E_N)\lra \Hom(G[m],\hocolim E_i)=\colim \Hom(G[m],E_i)\]
is an isomorphism. Since $m-1<m<-d$, the morphism $\Hom(G[m-1],\lambda)$ is also an isomorhpism. From the exact sequence obtained by applying $\Hom(G[m],-)$ to the above triangle $\xi_N$, we deduce that $\Hom(G[m],Y)=0$ for all $m<-d$, that is, $Y\in\T_{\G}^{\leq d}$. Note that $E_N\in\T^c$. The triangle $\xi_N$ shows that $X=\hocolim E_i\in \T^c*\T_{\G}^{\leq d}$. It follows that $X\in\bigcap_{d\in\mathbb{Z}}\T^c*\T_{\G}^{\leq d}=\T_c^-$.

Next we prove that $\T_c^-\subseteq \precompletionc{\T^c}{\M^{\G}}$. Take any object $X\in\T_c^-$,  and set $r_1=1$. There is a triangle 
\[E_1\lraf{f_1} X\lraf{g_1} X^{\leq r_1}\lra E_1[1]\]
in $\T$ such that $E_1\in\T^c$, $X^{\leq r_1}\in\T_{\G}^{\leq r_1}$. There is an integer $s_1$ such that $\Hom(G, E_1[k])=0$ for all $k>s_1$, that is, $E_1\in \langle \G\rangle^{[-s_1,+\infty)}$.  Let $r_2=\min\left\{ r_1,-s_1\right\}-1$. There is a triangle 
\[E_2\lraf{f_2} X\lraf{g_2} X^{\leq r_2}\lra E_2[1]\]
in $\T$ such that $E_2\in\T^c$, $X^{\leq r_2}\in\T_{\G}^{\leq r_2}$. Thus $\Hom(E_1, X^{\leq r_2})=0$ and consequently we have a commutative diagram
\[\xymatrix{
E_1\ar[r]^{f_1}\ar[d]^{\alpha_1} & X\ar[r]^{g_1}\ar@{=}[d] & X^{\leq r_1}\ar[r]\ar[d]^{\beta_1} & E_1[1]\ar[d]^{\alpha_1[1]}\\
E_2\ar[r]^{f_2} & X\ar[r]^{g_2} & X^{\leq r_2}\ar[r] & E_2[1]\\
}\]
Moreover, there is an isomorphism 
\[\cone(\alpha_1)\simeq \cone(\beta_1)[-1]\in \T^c\cap\T_{\G}^{\leq r_1-1}=\langle\G\rangle^{(-\infty,r_1-1]}\]
by the Octahedral Axiom. Repeating this process, we can inductively construct a commutative diagram

\[\xymatrix{
	E_1\ar[r]^{\alpha_1}\ar[d]_{f_1}&E_1\ar[r]^{\alpha_2}\ar[d]_{f_2}&E_3\ar[r]\ar[d]_{f_3}&\cdots\\
	X\ar@{=}[r]\ar[d]_{g_1}&X\ar@{=}[r]\ar[d]_{g_2}&X\ar@{=}[r]\ar[d]_{g_3}&\cdots\\
	X^{\leq r_1}\ar[r]^{\beta_1}\ar[d]&X^{\leq r_2}\ar[r]^{\beta_2}\ar[d]&X^{\leq r_3}\ar[r]\ar[d]&\cdots\\
	E_1[1]\ar[r]&E_2[1]\ar[r]&E_3[1]\ar[r]&\cdots
	}\]
in $\T$ such that the following hold.
\begin{itemize}
    \item Each column is a triangle and $E_i\in \T^c$ for all $i\geq 1$.
    \item $r_{i+1}\leq r_i-1$ for all $i\geq 1$ and $X^{\leq r_i}\in\T_{\G}^{\leq r_i}$ for all $i\geq 1$.
    \item $\cone(\alpha_i)\simeq \cone(\beta_i)[-1]\in\langle\G\rangle^{(-\infty,r_i-1]}$ for all $i\geq 0$. 
\end{itemize}
Since $r_i$ is descending,  $E_1\ra E_2\ra \cdots$ is a Cauchy sequence in $\T^c$. By Lemma \ref{lem-desc-hocolim}, one has $\hocolim X^{\leq r_i}=0$. Consequently $\hocolim E_i=\hocolim X=X$. Hence $X\in\precompletionc{\T^c}{\M^{\G}}$. 

For each $X\in\T_c^b=\T_c^-\cap\T^+$, there is an integer $n$ such that $X\in \T_{\G}^{\geq n}$. This implies that \(\Hom\big(\left\langle \G\right\rangle^{(-\infty,n-1]},X\big)=0\). Together with $X\in\T_c^-=\precompletionc{\T^c}{\M^{\G}}$, we deduce that  $X\in\completionc{\T^c}{\M^{\G}}$. Conversely, for each $X\in\completionc{\T^c}{\M^{\G}}$, we have $X\in\T_c^-$ since $\completionc{\T^c}{\M^{\G}}\subseteq \precompletionc{\T^c}{\M^{\G}}=\T_c^-$. By the definition of $\completionc{\T^c}{\M^{\G}}$, there is an   integer $D$ such that $\Hom(\left\langle \G\right\rangle^{(-\infty,-D]},X)=0$. Particularly $\Hom(\G[i], X)=0$ for all $i\geq D$. Hence $X\in \T^+$ and consequently $X\in\T_c^-\cap\T^+=\T_c^b$.
\end{proof}

One may ask whether $\T_c^-$ in Proposition \ref{prop-silt-completion} is a triangulated subcategory of $\T$. Actually, the proof of \cite[Lemma 2.9]{Neeman2026aa} still works in this case, that is, the subcategory $\T_c^-$ is a triangulated subcategory of $\T$. The objects in $\T_c^-$ has a crucial property as shown in the following lemma, which will be used later on.
 
\begin{Lem}\label{lem-tcf-coprod}
Keep the notations in Proposition \ref{prop-silt-completion}. For an object $X\in\T_c^-$ and a family of objects $\{Y_i\}_{i\in I}$ in $\T_{\G}^{\geq l}$ for some fixed $l$, the canonical morphism
\[\coprod_{i\in I}\Hom(X,Y_i)\lra \Hom\left(X,\coprod_{i\in I}Y_i\right)\]
is an isomorphism. 
\end{Lem}
\begin{proof}
By definition, there is a triangle $Z[-1]\lra C\lra X\lra Z$ with $Z\in\T_{\G}^{\leq l-1}$ and $C\in\T^c$. The lemma follows easily by applying $\Hom(-,\coprod_{i\in I}Y_i)$ to the triangle, since $\coprod_{i\in I}Y_i\in\T_{\G}^{\geq l}$ and thus $\Hom(Z[m], \coprod_{i\in I}Y_i)=0$ for $m=0,-1$. 
\end{proof}

If $\T=\D[\Modcat{R}]$ is the derived category of a ring $R$ and $\G=\add(R)$, then $\T_c^{-}=\Kf[\projcat{R}]$ and $\T_c^b=\Kfb[\projcat{R}]$. If $R$ is derived equivalent to a right coherent ring, then it follows from a result \cite[Corollary 4.3]{Marks2023} (see also \cite[Proposition 7.6]{Saorin2022}) that the $t$-structure $(\T_{\G}^{\leq 0},\T_{\G}^{\geq 0})$ restricts to $\T_c^b=\Kfb[\projcat{R}]$ if and only if $R$ is also right coherent. We have the following general quesiton. 

\medskip 
\noindent{\bf Question.} {\em In the setting of Proposition \ref{prop-silt-completion}, when does the $t$-structure restrct to $\T_c^-$ or $\T_c^b$}\,?

\medskip 
Our answer to this question is Theorem \ref{theo-main-intro}. Particularly, the $t$-structure in Proposition \ref{prop-silt-completion} restricts to a $t$-structure on $\T_c^-$ if and only if $\G$ is contravariantly finite in $\T_c^-$. The rest of this section is devoted to giving a proof of Theorem \ref{theo-main-intro}. For simplicity, we keep the notations in Proposition \ref{prop-silt-completion} for the rest of this section.

\begin{Lem}
\label{Lem-contr-in-negpart}
Every object in $\T_c^{-}\cap\T_{\G}^{\leq 0}$ admits a right $\G$-approximation.  
\end{Lem}
\begin{proof}
Let $X\in \T_c^-\cap\T_{\G}^{\leq 0}$. Since $X\in\T_c^-$, there is a triangle \[B[-1]\lra A\lraf{f} X\lra B\] with $B\in\T_{\G}^{\leq -1}$ and $A\in\T^c$. This gives an exact sequence in $\Modcat{\G}$
\[\Hom(-,A)|_{\G}\lraf{(-,f)}\Hom(-,X)|_{\G}\lra \Hom(-,B)|_{\G}=0,\]
showing that $(-,f)$ is an epimorphism. Now both $X$ and $B[-1]$ belong to $\T_{\G}^{\leq 0}$. It follows that $A\in\T_{\G}^{\leq 0}\cap \T^c$. Since $\G$ is a silting subcategory of $\T^c$, we have 
\[\T_{\G}^{\leq 0}\cap \T^c=\langle\G\rangle^{(-\infty,0]}=\bigcup_{n\geq 0} \add(\G*\G[1]*\cdots *\G[n]).\] 
Thus, there is a positive integer $N$ such that $A\in \add(\G[0]*\G[1]*\cdots*\G[N])$. Let $G\raf{g} A'\ra V\ra G[1]$ be a triangle such that $G\in \G$, $V\in \G[1]*\G[2]*\cdots*\G[N]$ and $A'$ has $A$ as a direct summand. This induces an exact sequence in $\Modcat{\G}$
\[\Hom(-,G)|_{\G}\lraf{(-,g)}\Hom(-,A')|_{\G}\lra \Hom(-,V)|_{\G}=0.\]
Hence $(-,g)$ is an epimorphism. Let $\pi: A'\ra A$ be the canonical projection. Then the composition 
 
\[(-,g)\circ (-,\pi)\circ (-,f)=(-,g\pi f): \Hom(-,G)|_{\G}\lra\Hom(-,X)|_{\G}\] is epic with $G\in \G$, that is, the morphism $g\pi f: G\ra X$ is a  right $\G$-approximation. 
\end{proof}

\begin{Prop}\label{prop-res2contr}
If the $t$-structure  $\left(\T^{\leq0}_{\G},\T^{\geq0}_{\G}\right)$ on $\T$ restricts to a $t$-structure on $\U=\T_c^-$ or $\T_c^b$, then every object in $\U$ admits a right $\G$-approximation.  
\end{Prop}
\begin{proof}
Since the $t$-structure  $\left(\T^{\leq0}_{\G},\T^{\geq0}_{\G}\right)$ on $\T$ restricts to $\U$, for each object $X\in\U$, there is a triangle
\[A\lraf{f} X\lra B\lra A[1]\]
in $\U$ such that $A\in\U\cap\T_{\G}^{\leq 0}$ and $B\in\U\cap\T_{\G}^{\geq 1}$. The triangle induces an exact sequence in $\Modcat{\G}$
\[\Hom(-,A)|_{\G}\lraf{(-,f)}\Hom(-,X)|_{\G}\lra \Hom(-,B)|_{\G}=0,\]
showing that $(-,f)$ is an epimorphism. Since $A\in\U\cap\T_{\G}^{\leq 0}$, by Lemma \ref{Lem-contr-in-negpart}, $A$ admits a right $\G$-approximation $g: G\ra A$ with $G\in \G$. Then the composition $fg: G\ra X$ is a right $\G$-approximation of $X$. 
\end{proof}

The following lemma is well-known, for the convenience of the reader, we provide a proof. 
\begin{Lem}\label{lem-stack-approximation}
Let $\T$ be a triangulated category and let $\A,\mathscr{C}$ be full subcategories of $\T$. Given a commutative diagram 
	\[\xymatrix{
		A\ar[r]^{f}\ar@{=}[d]&B\ar[r]^{s}\ar[d]&D\ar[r]\ar[d]&A[1]\ar@{=}[d]\\
		A\ar[r]&C\ar[r]\ar[d]^{g}&E\ar[r]\ar[d]^{h}&A[1]\\
		&F\ar@{=}[r]\ar[d]&F\ar[d]^{u}&\\
		&B[1]\ar[r]&D[1]&
	}\]
	in $\T$ such that all rows and columns are triangles in $\T$. If $f:A\to B$ is a right $\A$-approximation and  $g:C\to F$ is a right $\mathscr{C}$-approximation, then $h:E\to F$ is a right $\mathscr{C}*\left(\A[1]\right)$-approximation. Particularly, suppose that $\mathscr{V}$ is a full subcategory of $\T$ closed under taking cones of morphisms in $\mathscr{V}$ and that every object in $\mathscr{V}$  admits a right $\A$-approximation and a right $\mathscr{C}$-approximation. Then every object in $\mathscr{V}$ admits a $(\mathscr{C}*\A)$-approximation. 
\end{Lem}

\begin{proof}
	Given $Z\in\mathscr{C}*(\A[1])$, we shall show that for any morphism $r:Z\to F$, $r$ factorizes through $h:E\to F$. By definition, there is a triangle
	\[A'\lraf{\alpha}C'\lraf{\beta}Z\lraf{\gamma} A'[1]\]
	with $C'\in\mathscr{C}$ and $A'\in\A$. Since $g:C\to F$ is a right $\mathscr{C}$-approximation, there exists $q:C'\to C$ such that $r\beta=gq$. Hence there is a commutative diagram 
	\[
	\xymatrix{
		A'\ar[r]^{\alpha}\ar[d]^{p}&C'\ar[r]^{\beta}\ar[d]^{q}&Z\ar[r]^{\gamma}\ar[d]^{r}&A'[1]\ar[d]^{p[1]}\\
		B\ar[d]^s\ar[r]&C\ar[d]\ar[r]^g&F\ar@{=}[d]\ar[r]&B[1]\ar[d]^{s[1]}\\
		D\ar[r]&E\ar[r]^h&F\ar[r]^u&D[1]
	}
	\]
	in $\T$ such that all rows are triangles. 
	Note that $A'\in\A$ and $f:A\to B$ is a right $\A$-approximation, there is a morphism $\mu:A'\to A$ such that $p=f\mu$. Hence $sp=sf\mu=0$ and thus $ur=(sp)[1]\circ \gamma=0$. It follows that $r$ factorizes through $h$. 
	%If both $\A$ and $\mathscr{C}$ are contravariantly finite in $\T$, then so are $\mathscr{C}[n]$ and $\A[m-1]$. The discussion above shows that every object in $\T$ admits a right $(\mathscr{C}[n])*(\A[m])$-approximation, that is, $(\mathscr{C}[n])*(\A[m])$ is contravariantly finite in $\T$. 
\end{proof}

\begin{Prop}\label{prop-slice-and-stack-silting}
Suppose that $\G$ is contravariantly finite in $\T_c^-$ and that $X\in \T_c^-\cap \T_{\G}^{\leq n}$. Then there are families of triangles 
\[\slicesilting(X)=\left\{P_i[-i]\lraf{p_i}X^{\leq i}\lraf{\pi_i} X^{\leq i-1}\lra P_i[-i+1]\mid i\leq n\right\},\]
	\[\stacksilting(X)=\left\{Q_{i,j}\lraf{q_{i,j}} X^{\leq i}\lraf{c_{i,j}} X^{\leq j}\lra Q_{i,j}[1]\mid j< i\leq n\right\}\]
such that 
	\begin{itemize}
		\item[$(1)$] $X^{\leq n}=X$ and   $P_i[-i]\lraf{p_i}X^{\leq i}$ is a right $\G[-i]$-approximation for all $i\leq n$; 
		\item [$(2)$]  $X^{\leq i}\in  \T_c^-\cap \T_\G^{\leq i}$ for all $i\leq n$;
		\item [$(3)$] $c_{i,j}=\pi_{j+1}\pi_{j+2}\cdots \pi_i$  and  $Q_{i,j}\lraf{q_{i,j}}X^{\leq i}$ is a right $\left(\G[-i]*\G[-i+1]*\cdots*\G[-j-1]\right)$-approximation for all $j< i\leq n$.
	\end{itemize}
\end{Prop}

\begin{proof}
Set $X^{\leq n}=X$. Since $\G$ is contravariantly finite in $\T_c^-$, $\G[m]$ is contravariantly finite in $\T_c^-[m]=\T_c^-$ for all integers $m$.
Inductively, for each $i\leq n$, take a right $\G[-i]$-approximation $P_{i}[-i]\lraf{p_{i}}X^{\leq i}$ and complete the morphism into a triangle
	\[P_{i}[-i]\lraf{p_{i}}X^{\leq i}\lraf{\pi_{i}} X^{\leq i-1}\lra P_i[-i+1]\quad (\dagger).\]
This gives rise to a family of triangles 
	\[\slicesilting(X)=\{P_i[-i]\lraf{p_i}X^{\leq i}\lraf{\pi_i} X^{\leq i-1}\lra P_i[-i+1]\mid i\leq n\}\]
satisfying $(1)$. Now we prove that 
$\slicesilting(X)$ satisfies $(2)$. 
Since $X^{\leq n}=X\in\T_c^-\cap\T_{\G}^{\leq n}$, for each $i\leq n$, we can inductively assume that $X^{\leq i}\in\T_c^-\cap\T_{\G}^{\leq i}$. 
For each $G\in\G$ and $k\in\mathbb{Z}$, applying $\Hom(G[k], -)$ to $(\dagger)$ results in an exact sequence
	\[\Hom(G[k],P_{i}[-i])\lraf{(G[k],p_{i})}\Hom(G[k],X^{\leq i})\lra\Hom(G[k],X^{\leq i-1})\lra\Hom(G[k],P_i[-i+1]).\]
If $k<-i$, then  $\Hom(G[k],X^{\leq i})=0=\Hom(G[k],P_i[-i+1])$. Thus $\Hom(G[k],X^{\leq i-1})$ vanishes. If $k=-i$, then $(G[-i],p_i)$ is surjective and $\Hom(G[-i],P_i[-i+1])=0$. This implies that $\Hom(G[-i],X^{\leq i-1})=0$. Hence $X^{\leq i-1}\in \T_{\G}^{\leq i-1}$. Note that $\G[k]\subseteq \T^c\subseteq\T_c^-$ for all $k\in\mathbb{Z}$. It follows that $X^{\leq i-1}\in\T_c^-$. We thus conclude that $\slicesilting(X)$ satisfies $(2)$. 

Finally we construct $\stacksilting(X)$. For each $j<n$, we proceed by induction on $i$. If $i=j+1$, then, setting $Q_{j+1,j}=P_{j+1}[-j-1]$ and $q_{j+1,j}=p_{j+1}$, the triangle 
 \[Q_{j+1,j}\lraf{q_{j+1,j}}X^{\leq j+1}\lraf{\pi_{j+1}}X^{\leq j}\lra Q_{j+1,j}[1]\]
is in $\slicesilting(X)$ and $q_{j+1,j}$ is a right $\G[-j-1]$-approximation.  Suppose the following triangles 
	\[Q_{i,j}\lraf{q_{i,j}}X^{\leq i}\lraf{c_{i,j}}X^{\leq j}\lra Q_{i,j}[1],\quad j<i\leq k<n\]
	have been constructed such  that  $q_{i,j}$ is a right $\left(\G[-i]*\G[-i+1]*\cdots*\G[-j-1]\right)$-approximation and that $c_{i,j}=\pi_{j+1}\pi_{j+2}\cdots\pi_i$ for all $j<i\leq k$. The Octahedral Axiom gives the following commutative diagram
	\[\xymatrix{
		Q_{k,j}[-1]\ar[r]^{q_{k,j}[-1]}\ar@{=}[d]&X^{\leq k}[-1]\ar[r]^{c_{k,j}[-1]}\ar[d]&X^{\leq j}[-1]\ar[r]\ar[d]&Q_{k,j}\ar@{=}[d]\\
		Q_{k,j}[-1]\ar[r]&P_{k+1}[-k-1]\ar[r]\ar[d]^{p_{k+1}}&Q_{k+1,j}\ar[r]\ar[d]^{q_{k+1,j}}&Q_{k,j}\\	
		&X^{\leq k+1}\ar@{=}[r]\ar[d]^{\pi_{k+1}}&X^{\leq k+1}\ar[d]^{c_{k+1,j}}&\\
		&X^{\leq k}\ar[r]^{c_{k,j}}&X^{\leq j}&
	}\]
such that all rows and columns are triangles.  Since \[Q_{k,j}[-1]\lraf{q_{k,j}[-1]}X^{\leq k}[-1]\] is a right $\left(\G[-k-1]*\G[-k]*\cdots*\G[-j-2]\right)$-approximation, and  \[P_{k+1}[-k-1]\lraf{p_{k+1}}X^{\leq k+1}\] is a right $\G[-k-1]$-approximation, by Lemma \ref{Lem-contr-in-negpart}, the morphism \[Q_{k+1,j}\lraf{q_{k+1,j}}X^{\leq k+1}\] is a right $\left(\G[-k-1]*\G[-k]*\cdots*\G[-j-1]\right)$-approximation. Moreover,  \[c_{k+1,j}=c_{k,j}\pi_{k+1}=\pi_{j+1}\cdots\pi_k\pi_{k+1}.\] Continuing this process, one can construct  a family of triangles
	\[\stacksilting(X)=\{Q_{i,j}\lraf{q_{i,j}} X^{\leq i}\lraf{c_{i,j}} X^{\leq j}\lra Q_i[1]\mid j< i\leq n\}\]
	satisfying condition $(3)$.
\end{proof}

\medskip 
With Proposition \ref{prop-slice-and-stack-silting}
, we can prove the following result. \begin{Prop}
\label{prop-res-co-t}
If $\G$ is contravariantly finite in $\T_c^-$, then $(\langle\G\rangle^{[0,+\infty)},\T_c^-\cap\T_{\G}^{\leq 0})$ is a  co-$t$-structure on $\T_c^-$ with co-heart $\G$.
\end{Prop}

\begin{proof}
It is clear that $\left\langle \G\right\rangle^{[0,\infty)}[-1]\subset\left\langle \G\right\rangle^{[0,\infty)}$ and  $\left(\T_c^-\cap\T_\G^{\leq0}\right)[1]\subset\T_c^-\cap\T_\G^{\leq0}$. Since $\Hom(\G,\G[m])=0$ for all $m>0$, we have
	\[\langle \G\rangle^{[0,\infty)}=\bigcup_{s\geq 0}\add(\G[-s]*\G[-s+1]*\cdots*\G).\]
Since $\T_{\G}^{\leq 0}[1]\subseteq (\G[m])^{\perp}$ for all $m\leq 0$, we deduce that $\T_{\G}^{\leq 0}[1]\subseteq\left(\langle\G\rangle^{[0,\infty)}\right)^{\perp}$. 

It is obvious that $\langle \G\rangle^{[0,\infty)}$ is closed under direct summands. The $t$-structure $(\T_{\G}^{\leq 0},\T_{\G}^{\geq 0})$ on $\T$ implies that $\T_{\G}^{\leq 0}$ is closed under direct summands.   Hence $\T_c^-\cap\T_{\G}^{\geq 0}$ is closed under direct summands in $\T_c^-$.

For each $X\in\T_c^-$,  we may assume $X\in\T_\G^{\leq n}$ for some  $n\geq0$. By Proposition \ref{prop-slice-and-stack-silting}, letting $X^{\leq n}=X$,  one can construct a triangle 
 \[Q_{n,-1}\lraf{q_{n,-1}}X^{\leq n}\lraf{c_{n,-1}}X^{\leq-1}\lra Q_{n,-1}[1]\]
such that $Q_{n,-1}\in\G[-n]*\G[-n+1]*\cdots*\G\subset\left\langle \G\right\rangle^{[0,\infty)}$, and  $X^{\leq-1}\in\T_c^-\cap\T_\G^{\leq -1}=\left(\T_c^-\cap\T_{\G}^{\leq 0}\right)[1]$.
	
By Definition \ref{def-cot}, we have shown that $\left(\langle \G\rangle^{[0,\infty)},\T_c^-\cap\T_\G^{\leq0}\right)$ is a co-$t$-structure on $\T_c^-$. The co-heart of this co-$t$-structure is $\langle\G\rangle^{[0,\infty)}\cap (\T_c^-\cap\T_\G^{\leq 0})=\G$.
\end{proof}

The next lemma concerns the sequences in $\T_c^-$ for which the homotopy colimit remains in $\T_c^-$. This may be compared to the theorem in analysis stating that the limit of a Cauchy sequence in a complete metric space also lies within the space.

\begin{Lem}\label{lem-big-Cauchyseq}
Suppose that $X_*: X_1\lraf{f_1} X_2\lraf{f_2}X_3\lra\cdots$ is a sequence in $\T_c^-$ such that, for each $n\geq 1$, the cone $\cone(f_n)\in\T_{\G}^{\leq \theta(n)}$ for some strictly descending function $\theta$. Then $\hocolim X_n\in\T_c^-$. 
\end{Lem}
\begin{proof}
First one can construct a triangle $E_1\lra X_1\lra Y_1\lra E_1[1]$ with $E_1\in\T^c$ and $Y_1\in\T_{\G}^{\leq \theta(1)}$. Since $E_1\in\T^c$, there is some positive itneger $d_1$ such that
$E_1\in\add(\G[-d_1]*\cdots *\G[d_1])$. It follows that $\Hom(E_1, Y)=0$ for all $Y\in\T_{\G}^{\leq -d_1-1}$. Based on the choice of $d_1$, we construct a triangle $E_1\lra X_2\lra Y_2\lra E_2[1]$ with $Y_2\in\T_{\G}^{\leq\min\{\theta(2),-d-1\}}$. Then $\Hom(E_1,Y_2)=0$, and there is a commutative diagram 
\[\xymatrix{
E_1\ar[r]\ar[d]^{h_1} & X_1\ar[r]\ar[d]^{f_1} & Y_1\ar[r]\ar[d]^{g_1} & E_1[1]\ar[d]^{h_1[1]}\\
E_2\ar[r] & X_2\ar[r] & Y_2\ar[r] & E_2[1].\\
}\] 
Since $\theta(2)<\theta(1)$ and $Y_1\in\T_{\G}^{\leq \theta(1)}$, we deduce that $\cone(g_1)\in\T_{\G}^{\leq \theta(1)-1}$. Together with the assumption that $\cone(f_1)\in\T_{\G}^{\leq \theta(1)}$, it follows that $\cone(h_1)\in\T_{G}^{\leq \theta(1)}\cap\T^c=\langle \G\rangle^{(-\infty,\theta(1)]}$. Since $Y_2\in\T_{\G}^{\leq \theta(2)}$, we can continue this process and inductively construct a commutative diagram 
\[\xymatrix{
E_1\ar[r]\ar[d]^{h_1} & X_1\ar[r]\ar[d]^{f_1} & Y_1\ar[r]\ar[d]^{g_1} & E_1[1]\ar[d]^{h_1[1]}\\
E_2\ar[r]\ar[d]^{h_2} & X_2\ar[r]\ar[d]^{f_2} & Y_2\ar[r]\ar[d]^{g_2} & E_2[1]\ar[d]^{h_2[1]}.\\
\vdots & \vdots & \vdots & \vdots 
}\] 
such that all rows are triangles, $\cone(h_n), Y_n\in\T_{\G}^{\leq \theta(n)}$ for all $n\geq 0$. Taking homotopy colimit, we get a triangle 
\[\hocolim E_n\lra \hocolim X_n\lra \hocolim Y_n\lra (\hocolim E_n)[1]\]
The sequence $E_*$ is a Cauchy sequence in $\T^c$ by definition, this implies that $\hocolim E_n\in\T_c^-$. For each integer $k$ and $G\in\G$, there is an isomorphism 
\[\Hom(G[k],\hocolim Y_n)\simeq \colim \Hom(G[k], Y_n)\simeq 0\]
since $\Hom(G[k],Y_n)=0$ when $\theta(n)+k<0$ which is always true if $n$ is sufficiently large. The generating property of $\G$ then implies that $\hocolim Y_n=0$. Thus $\hocolim X_n\simeq \hocolim E_n$ belongs to $\T_c^-$. 
\end{proof}

Now we are in a position to consider the restriction of the $t$-structure $(\T_{\G}^{\leq 0},\T_{\G}^{\geq 0})$. 

\begin{Prop}
\label{prop-rest-tstr}
 If $\G$ is contravariantly finite in $\T_c^-$, then $(\T_{\G}^{\leq 0},\T_{\G}^{\geq 0})$ restricts to   $\T_c^-$. 
\end{Prop}
\begin{proof}
 For each object $X\in\T_c^-$, the $t$-structure $(\T_{\G}^{\leq 0},\T_{\G}^{\geq 0})$ provide us a  triangle $X^{\leq 0}\lra X\lra X^{\geq 1}\lra X^{\leq 0}[1]$ with $X^{\leq 0}\in\T_{\G}^{\leq 0}$ and $X^{\geq 1}\in\T_{\G}^{\geq 1}$. The main idea of this proof is to show that $X^{\geq 1}\in\T_c^b\subseteq\T_c^-$. Once this were done, it follows immediately that $X^{\leq 0}\in\T_c^-$ since $\T_c^-$ is a triangulated subcategory of $\T$.

 Since $X\in\T_c^-$, there is a positive integer $n$ such that $X\in\T_{\G}^{\leq n}\cap\T_c^-$. We fix the number $n$ in the proof. Set $X_1=X$. We shall construct a sequence of morphisms 
\[X_1\lraf{\lambda_1}X_2\lraf{\lambda_2}X_3\lra\cdots\]
in $\T_{\G}^{\leq n}\cap\T_c^-$ satisfying the following properties.

\smallskip 
\begin{itemize}
    \item[(i)] Each $\lambda_i, i\geq 1$ fits into a triangle
\[\xi_i:\quad R_{i}[i-1]\lraf{r_i}X_{i}\lraf{\lambda_i}X_{i+1}\lra R_i[i]\]
such that $r_i$ is a right $\G[i-1]$-approximation.
\item[(ii)] For $G\in\G$,   $\Hom(G[k],\lambda_i)$ is an isomorphism for $k\leq i-2$, and $\Hom(G[k],X_i)=0$ for $0\leq k\leq i-2$. 
\end{itemize}

\smallskip 
\noindent 
By induction, for each $i\geq 1$, assume that  $X_i\in\T_{\G}^{\leq n}\cap\T_c^-$, take a right $\G[i-1]$-approximation $R_{i}[i-1]\lraf{r_{i}}X_{i}$ and extend it into a triangle
\[\xi_i:\quad R_{i}[i-1]\lraf{r_{i}}X_{i}\lraf{\lambda_{i}}X_{i+1}\lra R_{i}[i].\]
Since $R_{i}\in\G$ and $i\geq 1$, it is clear that $R_{i}[i]$ belongs to $\T_{\G}^{\leq 0}\cap\T^c$ which is contained in $\T_{\G}^{\leq n}\cap\T_c^-$. Hence $X_{i+1}\in \T_{\G}^{\leq n}\cap\T_c^-$.

We next show that each triangle $\xi_{i}$ satisfies the above properties (i) and (ii). For each $G\in \G$, applying $\Hom(G[k],-)$ to the triangle $\xi_{i}$ results in an exact sequence 
\[
\Hom(G[k],R_{i}[i-1])\lraf{\Hom(G[k],r_{i})} \Hom(G[k],X_{i})\lraf{\Hom(G[k],\lambda_{i})} \Hom(G[k],X_{i+1})\ra \Hom(G[k],R_{i}[i]).\]
If $k\leq i-2$, then both $\Hom(G[k],R_{i}[i-1])$ and $\Hom(G[k],R_{i}[i])$ vanish, and consequently $\Hom(G[k], \lambda_{i})$ is an isomorphism. Finally, we  prove that $\Hom(G[k],X_i)=0$ for $0\leq k\leq i-2$ by induction on $i$. Clearly, we can assume that $i\geq 2$, otherwise there is nothing to prove. If $k=i-2$, then $\Hom(G[k],R_{i}[i-1])=0$ and $\Hom(G[k],r_{i-1})$ is surjective since $r_{i-1}$ is a right $\G[i-2]$-approximation. Hence $\Hom(G[k],X_i)=0$. This particularly implies that the case for $i=2$ holds. Now we inducttively assume that the case for $i\geq 2$ is true and assume that $0\leq k< (i+1)-2=i-1$. Then $\Hom(G[k],R_{i}[i])=0$ and by induction $\Hom(G[k], X_i)=0$. It follows that $\Hom(G[k],X_{i+1})=0$. This finishes the induction.	 

Since $\G\subseteq \T^c$, for each $G\in \G$ and  $k\in\mathbb{Z}$, there is an isomorphism 
\[\Hom(G[k],\hocolim X_i)\simeq \colim \Hom(G[k],X_i).\]
By property (ii) above, $\Hom(G[k],\lambda_i)$ is an isomorphism for $i>k+1$. It follows that 
\[\colim \Hom(G[k],X_i)\simeq \Hom(G[k],X_{k+2}),\]
which is zero if $k\geq 0$ by property (ii). Hence $\Hom(G[k],\hocolim X_i)=0$ for all $k\geq 0$. This means that $\hocolim X_i\in\T_{\G}^{\geq 1}$.

By definition, we have the following commutative diagram 
\[\xymatrix{
X_1\ar[r]^{\lambda_1}\ar[rd]^{\theta_1} & X_2\ar[r]^{\lambda_2}\ar[d]^{\theta_2} & X_3\ar[r]\ar[ld]^{\theta_3} &\cdots\\
& \hocolim X_i
}\]
For each $G\in\G$ and $k<0$, applying $\Hom(G[k], -)$ gives the following commutative diagram
\[\xymatrix{
\Hom(G[k],X_1)\ar[r]^{\simeq }\ar[rd]_{\Hom(G[k],\theta_1)} & \Hom(G[k],X_2)\ar[r]^{\simeq }\ar[d]^{\simeq } & \Hom(G[k],X_3)\ar[r]\ar[ld]^{\simeq} &\cdots\\
& \Hom(G[k],\hocolim X_i)\ar[r]_{\simeq} & \colim \Hom(G[k],X_i),
}\]
where the isomorphisms follow from the property (ii) above. This implies that $\Hom(G[k],\theta_1)$ is an isomorphism. Extending the morphism $\theta_1$ into a triangle
\[W\lra X_1\lraf{\theta_1}\hocolim X_i\lra W[1]\]
and applying $\Hom(G[k],-)$ to this triangle, the resulting long exact sequence shows that $\Hom(G[k],W)=0$ for all $k<0$. Hence $W\in\T_{\G}^{\leq 0}$. 

It remains to prove that $\hocolim X_i\in \T_c^b$. We already have $\hocolim X_i\in\T_{\G}^{\geq 1}$. It suffices to show that $\hocolim X_i\in\T_c^-$. This follows from Lemma \ref{lem-big-Cauchyseq} easily. 
\end{proof}

Next we consider the heart of the restricted $t$-structure. 

\begin{Prop}
\label{prop-res-heart}
 Assume that $\G$ is contravariantly finite in $\T_c^-$. Let $\mathcal{H}$ be the heart of
  the $t$-structure $(\T_{\G}^{\leq 0}\cap\T_c^-,\T_{\G}^{\geq 0}\cap\T_c^-)$ on $\T_c^-$. Then there are two equivalences of abelian  categories. 
    \[\arraycolsep=3pt
    \begin{array}{rcccl}
    \modcat{\G} &\llaf{\simeq} &\mathcal{H}&\lraf{\simeq}& \modcat{H^0(\G)}\\
    \Hom(-,X)|_{\G} &\mapsfrom & X &\mapsto & \Hom(-,X)|_{H^0(\G)}
    \end{array}\]
Here $H^0$ is the functor $\tau^{\geq 0}\circ \tau^{\leq 0}$ induced by the $t$-structure. 
\end{Prop}

\begin{proof}
Since $\tau^{\leq 0}$ is isomorphic to the  identity functor on $\T_{\G}^{\leq 0}$, the functor $\tau^{\geq 0}$ coincide with $H^0$ on $\T_{\G}^{\leq 0}$. Moreover, for $X\in\T_{\G}^{\leq 0}$ and $Y\in\T_{\G}^{\geq 0}$, the natural homomorphism 
\[\tau^{\geq 0}: \Hom(X, Y)\lra \Hom(H^0(X), Y)\]
is an isomorphism. Actually, this can be obtained by applying $\Hom(-,Y)$ to the canonical triangle 
\[\xi_X: \quad \tau^{\leq -1}X\lraf{\epsilon} X\lraf{\mu} \tau^{\geq 0}X\lra (\tau^{\leq -1}X)[1].\]
For two objects $G_1,G_2\in \G$, the triangles for $G_1,G_2$ gives rise to a commutative diagram 
\[\xymatrix@C=15mm{
\Hom(G_1,G_2)\ar[r]^{(-,\mu)}\ar[rd]_{\tau^{\geq 0}=H^0} & \Hom(G_1,\tau^{\geq 0}G_2)\\
& \Hom(\tau^{\geq 0}G_1,\tau^{\geq 0}G_2)\ar[u]^{(\mu,-)}
}\]
Both $(-,\mu)$ and $(\mu,-)$ here are isomorphisms. This follows by applying $\Hom(G_1,-)$ and ${}_{\G}(-,G_2)$ to the canonical triangles $\xi_{G_2}$ and $\xi_{G_1}$, respectively. It follows that $H^0$ in the above diagram is also an isomorphism. Hence $H^0=\tau^{\geq 0}: \G\lra H^0(\G)$ is an additive equivalence. Thus, we get an equivalence 
\[H: \Modcat{H^0(\G)}\lraf{\simeq} \Modcat{\G}, \quad F\mapsto F\circ H^0.\]

Since $\G$ is contravariantly finite in $\T_c^-$, it follows easily that $\Hom(-, X)|_{\G}$ is finitely presented in $\Modcat{\G}$ for all $X\in \T_c^-$, that is, $\Hom(-,X)|_{\G}\in\modcat{\G}$. Particularly, we have a diagram of functors which is commutative up to natural isomorphisms. 
\[\xymatrix{
\mathcal{H}\ar[r]^(.45){\Y^0}\ar[rd]_{\Y} & \modcat{H^0(\G)}\ar[d]^{H}\\
& \modcat{\G}
}\]
Here $\Y^0$ ($\Y$, respectively) sends $X\in \mathcal{H}$ to $\Hom(-,X)|_{H^0(\G)}$ ($\Hom(-,X)|_{\G}$, respectively). It suffices to show that $\Y^0$ is an exact equivalence.

For each short exact sequence $0\lra X\lra Y\lra Z\lra 0$ in $\mathcal{H}$ and $G\in\G$, applying $\Hom(G,-)$, one gets a commutative diagram with exact rows.
\[\xymatrix@C=5mm{
\Hom(G, Z[-1])\ar[r]\ar@{=}[d] & \Hom(G, X)\ar[r]\ar[d]^{\tau^{\geq 0}} & \Hom(G, Y)\ar[r]\ar[d]^{\tau^{\geq 0}} & \Hom(G, Z)\ar[r]\ar[d]^{\tau^{\geq 0}} & \Hom(G, X[1])\ar@{=}[d]\\
0\ar[r] & \Hom(H^0(G), X)\ar[r] & \Hom(H^0(G), Y)\ar[r] & \Hom(H^0(G), Z)\ar[r] & 0\\
}\]
The vertical maps induced by $\tau^{\geq 0}$ are all isomorphisms. Hence $H^0(G)$ is a projective object in $\mathcal{H}$ for all $G\in\G$ and both $\Y^0$ and $\Y$ above respectively are exact functors from $\mathcal{H}$ to $\modcat{H^0(\G)}$ and $\modcat{\G}$. Next, we show that $\mathcal{H}$ has enough projective objects and the projective objects are precisely thos objects in $H^0(\G)$. For any $X\in\mathcal{H}$, we can form a triangle $K\lra G_X\lraf{f} X\lra K[1]$ such that $f$ is a right $\G$-approximation. For each $G\in\G$, applying $\Hom(G,-)$ to this triangle implies that $\Hom(G, K[i])=0$ for all $i>0$, that is, $K\in\T_{\G}^{\leq 0}$. Thus, all terms $K, G_X, X$ and $K[1]$ in the above triangle are in $\T_{\G}^{\leq 0}$. Thus applying $\tau^{\geq 0}$, which coincide with $H^0$ on $\T_{\G}^{\leq 0}$,  to this triangle gives rise to a triangle $H^0(K)\lra H^0(G)\lraf{H^0(f)} X\lra H^0(K)[1]$. This means that the sequence 
\[0\lra H^0(K)\lra H^0(G_X)\lraf{H^0(f)} X\lra 0\]
is an exact sequence in $\mathcal{H}$. Hence $\mathcal{H}$ has enough projective objects and $\projcat{\mathcal{H}}=H^0(\G)$.  It follows that $\Y^0: \mathcal{H}\lra \modcat{H^0(\G)}$ is an equivalence.
\end{proof}

Now we are ready to give the proof of Theorem \ref{theo-main-intro}. 
\begin{proof}[Proof of Theorem \ref{theo-main-intro}]
    By Proposition \ref{prop-res2contr} and Proposition \ref{prop-rest-tstr}, (ii) and (iii) are equivalent. Clearly (ii) implies (i). Assume now that (i) holds. By definition, each object $X\in\T_c^-$ admits a triangle $C\raf{f} X\ra Y\ra C[1]$ with $C\in \T^c$ and $Y\in\T_{\G}^{\leq -1}$. This induces an exact sequence of functors 
    \[\Hom(-,C)|_{\G}\lraf{(-,f)} \Hom(-,X)|_{\G}\lra \Hom(-,Y)|_{\G}=0,\] 
showing that $(-,f)$ is an epimorphism in $\Modcat{\G}$. By (i), the object $C\in\T^c$ admits a right $\G$-approximation, say $r: G\lra C$ with $G\in\G$. Then the morphism $f\circ r: G\lra X$ is a right $\G$-approximation of $X$. Hence $\G$ is contravariantly finite in $\T_c^-$. This proves that (i), (ii) and (iii) are equivalent.

Suppose (iii) holds. For each $X\in\T_c^b$, the object $\tau^{\geq 1}X$ in the canonical triangle $\tau^{\leq 0}X\lra X\lra \tau^{\geq 1}X\lra (\tau^{\leq 0}X)[1]$ belongs to $\T_c^-\cap \T_{\G}^{\geq 1}\subseteq\T_c^b$. It follows that the $t$-structure $(\T_{\G}^{\leq 0}\cap\T_c^-, \T_{\G}^{\geq 0}\cap\T_c^-)$ can always further restrict to a $t$-structure on $\T_c^b$. Hence ${\rm (iii)}\Rightarrow {\rm (iv)}$. 

${\rm (iv)}\Rightarrow {\rm (v)}$ follows from Proposition \ref{prop-res2contr}.  The statement (v) implies (i) is clear if $\T^c\subseteq\T_c^b$. 

Now assume that the equivalent conditions (i)-(iii) hold.

Let $\U=\T_c^-$ or $\T_c^b$. Since $\G$ is a silting subcategory in $\T^c$, each object $C\in\T^c$ belongs to $\add(\G[-r]*\cdots*\G[r])$ for some positive integer $r$. This implies that $X\in\T_{\G}^{\leq r}$. By definition,  every object $X\in\U\subseteq \T_c^-$  is an extension of a compact object $C$ and some object $Y\in\T_{\G}^{\leq n}$. If $C\in\T_{\G}^{\leq r}$, then $X\in\T_{\G}^{\leq \max\{r,n\}}$. Hence 
$\U\subseteq \bigcup_{m\in\mathbb{Z}}\T_{\G}^{\leq m}\cap\U$ and the restriction of $(\T_{\G}^{\leq 0},\T_{\G}^{\geq 0})$ to $\U$ is an above-bounded $t$-structure. In case that $\U=\T_c^b\subseteq \T_{\G}^+$, the $t$-structure is also below-bounded. Hence the $t$-structure further restricts to a bounded $t$-structure on $\T_c^b$. Note that 
 \[\T_{\G}^{\leq 0}\cap\T_{\G}^{\geq 0}\cap\T_c^-\subseteq \T_{\G}^{\leq 0}\cap\T_{\G}^{\geq 0}\cap\T_c^b,\]
which must be equal. Thus the two restricted $t$-structure have the same heart. Together with Proposition \ref{prop-res-heart}, this proves (1).

It follows from Proposition \ref{prop-res-co-t} that  $(\langle\G\rangle^{[0,+\infty)}, \T_{\G}^{\leq 0})$ is a co-$t$-structure  on $\T_c^-$ with co-heart $\G$.  Suppose that $\T^c\subseteq\T_c^b$. For each object $X\in\T_c^b$, the co-$t$-structure on $\T_c^-$ gives a canonical triangle $C\lra X\lra Y\lra C[1]$ with $C\in\langle G\rangle^{[0,+\infty)}$ and $Y\in\T_c^-\cap\T_{\G}^{\leq 0}$. The assumption $\T^c\subseteq\T_c^b$ implies that $C\in\T_c^b$ and consequently $Y\in\T_c^b$. This shows that the triangle is actually in $\T_c^b$. Hence the co-$t$-structure further restricts  to a co-$t$-structure $(\langle\G\rangle^{[0,+\infty)},\T_c^b\cap\T_{\G}^{\leq 0})$ on $\T_c^b$ and the co-heart is still $\G$.  This proves (2).
\end{proof}

The following corollary is a consequence of Theorem \ref{theo-main-intro}.
\begin{Coro}
Let $R$ be a ring. Then the following are equivalent.
\begin{itemize}
    \item $R$ is right coherent.
    \item The standard $t$-structure of $\D[\Modcat{R}]$ restricts to $\Kf[\projcat{R}]$. 
    \item The standard $t$-structure of $\D[\Modcat{R}]$ restricts to $\Kfb[\projcat{R}]$. 
\end{itemize}
\end{Coro}

\begin{proof}
Set $\T=\D[\Modcat{R}]$ and $\G=\add(R)$. Then $\T^c=\Kb[\projcat{R}]$. Then the precompletion and completion of $\T^c$ under the metric $\M^{\G}$ induced by $\G$ are $\Kf[\projcat{R}]$ and $\Kfb[\projcat{R}]$. Let $\U$ to be $\Kf[\projcat{R}]$ or $\Kfb[\projcat{R}]$. It is easy to see that $\add(R)$ is contravariantly finite in $\U$ if and only if $H^0(\cpx{P})$ a finitely generated $R$-module for all $\cpx{P}\in\U$, if and only if $R$ is right coherent, that is, the kernel of each morphism between two finitely generated projective right $R$-modules is finitely generated. The corollary then follows from Theorem \ref{theo-main-intro} directly.
\end{proof}

\section{Weak generation of presilting objects}
In a compactly generated triangulated category $\T$ with coproducts, a classic result is that if a set $\mathcal{X}$ of compact objects weakly generates  $\T$, that is, for any non-zero  object $Y\in\T$, there is some object $X\in\mathcal{X}$ and integer $n$ such that $\Hom(X,Y[n])\neq 0$, then $\thick(\mathcal{X})=\T^c$. This means we need to test all objects in the big category $\T$.  A natrual question is, for some special cases, can we use certain smaller testing subcategories? 
Our answer to this question is the following theorem. 
\begin{Theo}\label{theo-testing}
Let $R$ be a commutative noetherian ring, and let   $\T$ be a compactly generated $R$-linear triangulated category with small coproducts. Suppose that $\T^c$ has a silting object $G$  such that $\Hom(G,G[n])$ is a finitely generated $R$-module for all $n\in\mathbb{Z}$, and that $M\in\T^c$ is a presilting object. Then $\thick(M)=\T^c$ provided that for each $X\in\T_c^-$, there exists some integer $n$ such that $\Hom(M, X[n])\neq 0$. 
\end{Theo}

For the proof of Theorem \ref{theo-testing}, we need the following lemma. 
\begin{Lem}\label{lem-noeth-fin}
	Keep the assumptions in Theorem \ref{theo-testing}. For each object $C\in\T^c$ and $X\in\T_c^-$, the Hom-space $\Hom(C,X)$ is a finitely generated as an $R$-module. Particularly, $\add(C)$ is contravariantly finite in $\T_c^-$.  
\end{Lem}
\begin{proof}
By assumption, there is some positive integer $r$ such that $C\in\add(G[-r]*\cdots*G[r])$. By the definition of $\T_c^-$, there is a triangle 
$C'\lra X\lra Y\lra C'[1]$ with $C'\in\T^c$ and $Y\in\T_{G}^{\leq -r-2}$. Applying $\Hom(C,-)$ to this triangle results in an exact sequence 
\[\Hom(C,Y[-1])\lra \Hom(C,C')\lra\Hom(C,X)\lra\Hom(C,Y)\]
Both $Y$ and $Y[-1]$ belong to $\T_{G}^{\leq -r-1}$. It follows that $\Hom(C,Y)=0=\Hom(C,Y[-1])$. Thus we get an isomorphism $\Hom(C,C')\simeq\Hom(C,X)$ of $R$-modules. Now $C'$ belongs to $\add(G[-n]*\cdots*G[n])$ for some positive integer $n$ since $C'\in\thick(G)$. The assumptions in Theorem \ref{theo-testing} then imply that 
$\Hom(C,C')$ is a finitely generated $R$-module. 
\end{proof}

\begin{proof}[Proof of Theorem \ref{theo-testing}]
For simplicity, we write $\Hom(-,-)$ for $\Hom(-, -)$. Clearly, we have 
\[M\in\add(G[-r]*\cdots *G[r])\]
for some positive $r$. It follows that $\Hom(M, M[k])=0$ for all $k>2r$. For each  $X\in\T_c^-$, we shall construct a sequence 
\[\xymatrix{
X_0\ar[r]^{\alpha_0}\ar[rd]_{\eta_0} & X_1\ar[r]^{\alpha_1}\ar[d]^{\eta_1} & X_2\ar[r]\ar[ld]^{\eta_2} & \cdots\\
& X
}\]
such that the following conditions are satisfied:
\begin{enumerate}
  \item Each $X_i$ is obtained by finitely many times of extensions by objects  of the form 
  \[\coprod_{m\geq d}M'_m[m],\]
where $M'_m\in\add(M)$ and $d\in\mathbb{Z}$. 
\item There is a strictly descending function $\theta$ such that $\cone(\alpha_n)\in\T_{G}^{\leq \theta(n)}$ for all $n\geq 0$. 
\item $\eta_{i+1}\alpha_i=\eta_i$ for all $i\geq 0$;
	\item 
	For each integer $n$ and each $i\geq 0$, the map
	\[\Hom(M[n], X_i)\lraf{(-,\eta_i)} \Hom(M[n], X)\]
	is surjective;		
	\item Let $\eta: \hocolim M_i\lra X$ be the induced map. Then the map 
	\[(-,\eta): \Hom(M[n], \hocolim X_i)\lra \Hom(M[n], X)\]
	is an isomorphism for all integers $n$.
	\item $\hocolim X_i\in \T_c^-$.
\end{enumerate}
Suppose that $X\in\T_G^{\leq n}$. Then $\Hom(M[k],X)=0$ for all $k<n+r$. We can define 
\[\min(X):=\min\{m\mid \Hom(M[m], X)\neq 0\}\]
For each $k\geq \min(X)$, let $B_{X,k}$ be a finite generating set of $\Hom(M[k], X)$ as an $R$-module. For simplcity, we write 
\[M_{X,k}=\coprod_{b\in B_{X,k}}M[k].\]
Then there is a triangle 
\[K_0\lraf{\beta_0} \coprod_{k\geq \min(X)} M_{X,k}\lraf{f_0} X\lra K_0[1],\]
where $f_0$ is induced by those $b\in B_{X,k}$, $k\geq n+r$. By construction,  $\Hom(M[m], f)$ surjective for integers $m$. Here $B_{X,k}$ can be chosen to be finite follows from  Lemma \ref{lem-noeth-fin}.  
Note that the object 
\[X_0:= \coprod_{k\geq \min(X)}M_{X,k}\]
belongs to $\T_c^-$ since it is the homotopy colimit of the Cauchy sequence 
\[B_{X,\min(X)}\lra \cdots \lra \coprod_{k=\min(X)}^{t}M_{X,k}\lra \coprod_{k=\min(X)}^{t+1}M_{X,k}\lra\cdots\]
in $\T^c$. It follows that $K_0\in\T_c^-$ and we can construct a morphism 
\[Z_0=\coprod_{k\leq \min(K_0)}M_{K_0,k}\lraf{g_0} K_0\]
such that $Z_0\in\T_c^-$ and $\Hom(M[m], g_0)$ surjective for integers $m$. Let $X_1=\cone(\beta_0 g_0)$, that is, there is a triangle 
\[Z_0\lraf{\beta_0g_0} X_0\lraf{\alpha_0} X_1\lra Z_0[1].\]
Thus, inductively, for each $i\geq 0$, we can construct a commutative diagram of triangles
\[\xymatrix{
Z_i\ar[r]\ar[d]^{g_i} & X_i\ar[r]^{\alpha_i}\ar@{=}[d] & X_{i+1}\ar[r]\ar[d]^{\eta_{i+1}} & Z_i[1]\ar[d]^{g_i[1]}\\
K_i\ar[r]^{\beta_i} & X_i\ar[r]^{\eta_i} & X \ar[r] &  K_i[1]
}\]
where $g_i$ is the morphism 
\[Z_i=\coprod_{k\geq \min(K_i)}M_{K_i,k}\lra K_i\]
satisfying that $\Hom(M[m],g_i)$ is surjective for all $m\in\mathbb{Z}$. Applying $\Hom(M[m],-)$ to this diagram results in a commutative diagram with exact rows and columns. 
\[\xymatrix@C=11mm{
&& \vdots\ar[d]\\
	&& \Hom(M[m], K_{i+1})\ar[d]^{(-,\beta_{i+1})}\\
\Hom(M[m], Z_i)\ar[r]^{(-, \beta_ig_i)}\ar@{->>}[d]_{(-,g_i)} & \Hom(M[m], X_i)\ar@{=}[d]\ar[r]^{(-,\alpha_i)} & \Hom(M[m], X_{i+1})\ar[r]\ar@{->>}[d]^{(-,\eta_{i+1})} & \Hom(M[m], Z_i[1])\\
\Hom(M[m], K_i)\ar[r]_{(-,\beta_i)} & \Hom(M[m], X_i)\ar@{->>}[r]_{(-,\eta_i)} & \Hom(M[m], X)
}\]
The morphisms with double heads are surjective for all integers $m$. Assume inductively that $Z_i$ is of the form 
\[Z_i=\coprod_{k\geq \min(K_0)+i} \coprod_{b\in B_k}M[k],\]
where $B_k$ is a finite set. This is clearly true for $i=0$. Now for each $m<\min(K_0)+i+1$, by the construction of $Z_i$, we have $\Hom(M[m],Z_i[1])=0$. It follows that $(-,\alpha_i)$ is surjective for such $m$. In the above diagram 
$$\Ker (-,\eta_i)=\Img(-,\beta_i)=\Img(-,\beta_ig_i)=\Ker(-,\alpha_i).$$
Hence $(-,\eta_{i+1})$ is an isomorphism for $m<\min(K_0)+i+1$. The vertical exact sequence the the above diagram then implies that $\Hom(M[m],K_{i+1})=0$ for all $m<\min(K_0)+i+1$, that is, $\min(K_{i+1})\geq \min(K_0)+i+1$. It follows that $Z_{i+1}$ is of the form 
\[Z_{i+1}=\coprod_{k\geq \min(K_0)+i+1}\coprod_{b\in B_k}M[k]\] 
by taking $B_k$ to be $B_{K_{i+1},k}$ which is finite. Thus we have constructed a sequence 
\[X_0\lraf{\alpha_0} X_1\lraf{\alpha_1}\cdots\quad (*)\]
such that the cone of $\alpha_n: X_n\lra X_{n+1}$ is 
\[Z_n[1]=\coprod_{k\leq \min(K_0)+n}M_{K_n,k}.\]
Note that $M\in\T_G^{\leq r}$, and $M[k]\in\T_G^{\leq r-k}$  
for all integers $k$. It follows that  
\[Z_n[1]\in\T_G^{\leq r-(\min(K_0)+n+1)}\] 
By Lemma \ref{lem-big-Cauchyseq}, we deduce that  $\hocolim X_n\in\T_c^-$.

Let $\eta: \hocolim X_n\lra X$ be the induced map. For each integer  $k$, we have the following commutative diagram with exact rows and columns. 
\[\xymatrix{
	& 0\ar[d] & 0\ar[d]\\
 & \coprod \Img (-,\beta_n)\ar[d] \ar[r]^{1-0}_{\simeq} & \coprod \Img (-,\beta_n)\ar[d]\\
 0\ar[r] & \coprod \Hom(M[k], X_n)\ar[r]^{1-\alpha_*}\ar[d]^{(-,\eta_n)} & \coprod \Hom(M[k], X_n)\ar[r]\ar[d]^{(-,\eta_n)} & \colim \Hom(M[k], X_n)\ar[r]\ar[d]^{(-,\eta)} & 0\\
 0\ar[r] & \coprod \Hom(M[k], X)\ar[d]\ar[r]^{1-{\rm shift}} & \coprod\Hom(M[k], X)\ar[d]\ar[r] & \Hom(M[k],X)\ar[r] & 0\\
 & 0 & 0
}\]
It follows that $(-,\eta)$ is an isomnorphism. In the triangle 
\[\hocolim X_n\lraf{\eta} X\lra \cone(\eta)\lra \hocolim X_n[1],\] 
$\Hom(M[k], \cone(\eta))=0$ for all integers $k$. Since both $\hocolim X_n$ and $X$ belong to $\T_c^-$, one has $\cone(\eta)\in\T_c^-$. By our assumption, we have $\cone(\eta)\simeq 0$, and thus $X\simeq \hocolim M_n$.

 Now let $C\in\T^c$, and let $X_n$ be as we have constructed above. Since $\cone(\eta)=0$, we have $C\simeq \hocolim X_n$. Applying $(C, -)$ to the triangle
\[\coprod X_n\lraf{1-\alpha} \coprod X_n\lraf{\theta} X\lra \coprod X_n[1]\]
results in an exact sequence
\[0\lra \Hom(C, \coprod X_n)\lraf{(C,1-\alpha)} \Hom(C, \coprod X_n)\lra \Hom(C, X)\lra 0\]
It follows that $1_C$ factors through finitely many $X_n$. We claim that every morphism from $C$ to $X_n$ factorizes through an object in $\add M[\mathbb{Z}]$. This is clear if $n=0$ since $X_0$ is a direct sum of objects in $M[\mathbb{Z}]$. Assume that $n>0$ and the claim holds for $X_{n-1}$. Consider the commutative diagram of triangles.
\[\xymatrix{
M'[-1]\ar[r]\ar[d]&C'\ar[r]^{u}\ar[d]^{v}& C\ar[d]^{f}\ar[r]^{g} & M'\ar[d]^{h}\\
Z_{n-1}\ar[r]& X_{n-1}\ar[r]^{\alpha_{n-1}}  & X_n\ar[r]  & Z_{n-1}[1] \\
}\]
The diagram is constructed as follows. Since $C\in\T^c$ and $Z_{n-1}[1]$ is a direct sum of objects in $M[\mathbb{Z}]$, the morphism from $C$ to $Z_{n-1}[1]$ factorizes through a finite direct sum of objects in $M[\mathbb{Z}]$. This gives rise to the morphisms $g$ and $h$ with $M'$ being a finite direct sum of objects in $M[\mathbb{Z}]$. Then we can complete $g$ to a morphism of triangles as above. Since $C'$ is an extension of $C$ by an object in $\add M[\mathbb{Z}]$, and $C$ is an object in $\T^c$, we have $C'\in\T^c$. By the induction hypothesis, every morphism from $C'$ to $X_{n-1}$ factorizes through an object in $\add M[\mathbb{Z}]$, say $v=\eta\theta$ with $\theta: C'\lra M''$ and $\eta: M''\lra X_{n-1}$, where $M''\in\thick(M)$. Complete $\theta$ into a triangle $M''[-1]\lra N\lraf{\xi} C'\lraf{\theta} M''\lra N[1]$. The Octahedral axiom then gives rise to a triangle $M'[-1]\lra N'\lraf{\xi'} M''\lraf{\theta'} M'$, where $N'$ is the cone of $u\xi$, which is in $\thick(M)$ since both $M'$ and $M''$ are in $\thick(M)$. Now 
\[fu\xi=\alpha_{n-1}v\xi=\alpha_{n-1}\eta\theta\xi=0.\]
It follows that $f$ factorizes through $N'$, which is in $\thick(M)$. 

Thus, we have proved that every morphism from $C$ to $M_n$ factorizes through $\thick(M)$. Recall that $1_C$ factorizes through the direct sum of finitely many $X_n$. It follows that $1_C$ factorizes through $\thick(M)$ and consequently $C\in\thick(M)$. Hence $\T^c=\thick(M)$. 
\end{proof}

\section{ST correspondences}

In the theory of triangulated categories, there are well-known correspondences between silting objects, $t$-structures and simple minded collections in varies situations (\cite{Aihara2012aa,Keller2013,Hernandez2013}). In \cite{Koenig2014aa}, Koenig and Yang established bijections between basic silting objects in $\Kb[\projcat{A}]$, co-$t$-structures on $\Kb[\projcat{A}]$, algebraic $t$-structures on $\Db[\modcat{A}]$ and simple minded collections in $\Db[\modcat{A}]$ for a finite-dimensional algebra $A$. This result was generalized to locally finite non-positive dg algebras ([\cite{Su2019,Fushimi2025}]), and the classical case is recovered by taking $\T=\D[\Modcat{A}]$, with $\T^c = \Kb[\projcat{A}]$ and $\T_c^b=\Db[\modcat{A}]$. It is natural to ask whether such bijections hold for $\T^c$ and $\T_c^b$ in our purely metric setting.

In this section, we assume that   $k$ is a field and  $\T$ is a compactly generated triangulated $k$-cateogry with small coproduct. Suppose that $\T^c$ has a silting object $G$ such that  $\Hom(G,G[n])$ is finite dimensional for all $n$. Consider the good metric on $\T^c$ induced by $G$, and let $\T_c^-$ and $\T_c^b$ respectively be the precompletion and completion with respect to this metric.

\begin{Lem}\label{lem-hom-fin}
Keep the notations above. The Hom-space $\Hom(X, Y)$ is finite dimensional over $k$ for all objects $X\in\T_c^-$ and $Y\in\T_c^b$. 
\end{Lem}
\begin{proof}
We can assume that $Y\in\T_{G}^{\geq m}$ for some integer $m$. By definition, there is a triangle $C\ra X\ra Z\ra C[1]$ with $C\in\T^c$ and $Z\in\T_{G}^{\leq m-2}$. Applying $\Hom(-,Y)$ to this triangle, one gets an exact sequence 
\[\Hom(Z,Y)\lra\Hom(X,Y)\lra \Hom(C,Y)\lra \Hom(Z[-1],Y).\] 
The condition $Z\in \T_{G}^{\leq m-2}$ implies that   $Z,Z[-1]\in \T_G^{\leq m-1}$. Hence $\Hom(Z,Y)=0=\Hom(Z[-1],Y)$. Thus $\Hom(X,Y)$ is isomorphic to $\Hom(C,Y)$ which is finite dimensional by Lemma \ref{lem-noeth-fin}. 
\end{proof}

\begin{Def}\label{def-SMC}
Let $\C$ be a triangulated caetgory over a field $k$ such that $\Hom_{\C}(X,Y)$ is finite dimensional for all $X, Y\in \C$. A collection $\mathcal{X}=\{X_1,X_2,\cdots,X_r\}$ of objects in $\C$ is called a {\em simple minded collection}, if   $\thick(\mathcal{X})=\C$ and there are division algebras $R_1,R_2,\cdots,R_r$ such that 
\[\Hom_{\C}(X_i,X_j[m])\simeq \begin{cases}
{R_i}, & i=j,m=0;\\
0, & m<0 \mbox{ or }m=0,i\neq j.	
\end{cases}\] 
\end{Def}
For instance, the simple modules over a finite dimensional aglebra $A$ form a simple minded collection in $\Db[\modcat{A}]$. Starting from a simple minded collection in $\Db[\modcat{A}]$ where $A$ is a finite dimensional symmetric aglebra, Rickard constructed in \cite{Rickard2002aa} a tilting object in $\Kb[\projcat{A}]$. If $A$ is not necessarily symmetric, Rickard's method actually gives an object $I\in\Kb[\injcat{A}]$
such that $\nu_A^-I$ is a silting object in $\Kb[\projcat{A}]$. In our setting, we do not have the Nakayama functor $\nu_A$. However, the construction still works. 

\begin{Lem}\label{lem-smc2cosilting}
Let $X_1,X_2,\cdots,X_r$ be a simple minded collection in $\T_c^b$. Then there are objects $I_1,I_2,\cdots, I_r$ in $\T$ such that 
\[\Hom(X_j, I_i[m])\simeq \begin{cases}
	{R_i}_{R_i}, & i=j, m=0,\\
	0, & \mbox{else},
\end{cases}\]
where $R_i$ is the endomorphism algebra of $X_i$.  Moreover, for each $i=1,2,\cdots,r$, the space $\coprod_{k\in\mathbb{Z}}\Hom(X[m], I_i)$ is finite dimensional for all $X\in\T_c^b$. 
\end{Lem}

\begin{proof}
We write $\mathcal{X}=\{X_1,X_2,\cdots,X_r\}$.  Clearly $\mathcal{X}\subseteq  \T_G^+$ and we assume that $\mathcal{X}\subseteq \T_G^{\geq l}$. 

Let us recall Rickard's construction.  For each $i=1,2,\cdots,r$, let $X_i^{(0)}=X_i$ and inductively define $X_i^{(n+1)}$ by the triangle
\[Z_i^{(n)}\lraf{\alpha_i^{(n)}} X_i^{(n)}\lraf{\beta_i^{(n)}} X_i^{(n+1)}\lra Z_i^{(n)}[1]\]
where 
\[Z_i^{(n)}(j,t)=\coprod_{B_i^{(n)}(j,t)} X_j[t], \mbox{ and }Z_i^{(n)}=\coprod_{1\leq j\leq r,t<0} Z_i^{(n)}(j,t),\]
$B_i^{(n)}(j,t)$ is a basis of $\Hom(X_j[t], X_i^{(n)})$ as vector space over $R_j$. Each $b\in B_i^{(n)}(j,t)$ is a morphism $b: X_j[t]\to X_i^{(n)}$,  and $\alpha_i^{(n)}$ is defined by the universal property of coproducts. In this construction, the following properties are satisfied. 
\begin{enumerate}
	\item $\Hom(X_j[t],\alpha_i^{(n)})$ is surjective for all $j=1,2,\cdots, r$ and $t<0$. 
	\item $\Hom(X_i[-1],\alpha_i^{(n)})$ is an isomnorphism. 
\end{enumerate}
The statement (1) is obvious from the construction. For $t<0$, $\Hom(X_i[-1], X_j[t])\neq 0$ unless $t=-1$ and $i=j$.  
Hence, by Lemma \ref{lem-tcf-coprod}, we have
\[\Hom(X_i[-1], Z_i^{(n)})\simeq\coprod_{1\leq j\leq r, t<0}\Hom(X_i[-1],Z_i^{(n)}(j,t))=\Hom(X_i[-1],\coprod_{B_i^{(n)}(i,-1)}X_i[-1])\]
If $\alpha_i^{(n)}\theta=0$ for some $\theta\in \Hom(X_i[-1], Z_i^{(n)})$. Then, for each $b\in B_i^{(n)}(i,-1)$, the $b$ component $\theta_{b}$ of $\theta$ satisfies $\alpha_i^{(n)}\theta_{b}=br_b$ for some $r_{b}\in R_i$, since $\Hom(X_i[-1], X_i[-1])\simeq R_i$. It follows that $r_{b}=0$ for all $b\in B_i^{(n)}(i,-1)$, and hence $\theta=0$. Thus (2) follows.  Set 
\[I_i=\hocolim X_i^{(n)}\]
for all $i=1,\cdots,r$. 

We claim that, for each $Y\in\T_c^-$, one has 
$\Hom(Y,I_i)\simeq \colim \Hom(Y,X_i^{(n)})$. Actually, Since $\mathcal{X}\subseteq \T_G^{\geq l}$, we have $Z_i^{(n)}\in\T_G^{\geq l}$ for all $i=1,\cdots,r$ and $n\geq 0$. Together with $X_i^{(0)}=X_i\in\T_G^{\geq l}$, inductively, we have $X_i^{(n)}\in\T_G^{\geq l}$ for all $i=1,\cdots,r$ and $n\geq 0$. Then it follows from Lemma \ref{lem-tcf-coprod} that 
\[\Hom(Y,I_i)\simeq \colim \Hom(Y,X_i^{(n)}).\]
Since $\Hom(X_j,Z_i^{(n)}[m])=0$ for all $m\leq 0$, applying $\Hom(X_j, -)$ to the triangle 
\[Z_i^{(n)}\lraf{\alpha_i^{(n)}} X_i^{(n)}\lraf{\beta_i^{(n)}} X_i^{(n+1)}\lra Z_i^{(n)}[1]\]
gives an isomorphism  $(X_j, \beta_i^{(n)}[m])$
for all $m<0$. It follows that 
\[\Hom(X_j, I_i[m])\simeq \colim \Hom(X_j, X_i^{(n)}[m])=0\]
for all $m<0$.  

The property (2) above implies that $\Hom(X_j, \beta_i^{(n)})$ is surjective for all $n$. If $i\neq j$, it follows that $\Hom(X_j, X_i^{(n)})=0$ for all $n$. Thus $\Hom(X_j, I_i)=0$ for all $j\neq i$. For $i=j$, we have $\Hom(X_i, X_i^{(0)})\simeq {R_i}_{R_i}$, the surjectivity of $\Hom(X_i, \beta_i^{(n)})$ implies that $\Hom(X_i, X_i^{(n)})\simeq {R_i}_{R_i}$ for all $n$. Hence $\Hom(X_i, I_i)={R_i}_{R_i}$. 

For $m>0$, the construction implies that $\Hom(X_j, \beta_i^{(n)}[m])=0$ for all $n$. Thus we deduce that  $\Hom(X_j, I_i[m])=0$ for all $m>0$.

Hence $\coprod_{k\in\mathbb{Z}}\Hom(X_j[m],I_i)$ is finite dimensional for all $i,j=1,2,\cdots,r$. Since $\thick(\mathcal{X})=\T_c^b$, for each $X\in\T_c^b$, the space $\coprod_{k\in\mathbb{Z}}\Hom(X[m],I_i)$ is finite dimensional for all $i$.
\end{proof}

Although there is no Nakayama functor in our setting, the representability theorem under certain additional condition recently given by Neeman can play the role. Applying \cite[Theorem 1.4]{Neeman2026aa} to our setting gives the following lemma. 
\begin{Lem}\label{lem-representability}
$(1)$ The image of the functor 
\[\T_c^b\lra Fun([\T^c]\opp,\Modcat{k}),\quad  X\mapsto \Hom(-, X)\]
consists of functors $F$ such that $\coprod_{m\in\mathbb{Z}}F(C[m])$ is finite dimensional for all $C\in\T^c$. 

$(2)$ Suppose that there is an object $G'\in\T_c^b$ and a positive integer $N$ such that $\T=\overline{\langle G'\rangle}_N$. Then the image of the functor $[\T^c]\opp\lra Fun(\T_c^b,\Modcat{k})$ sending $C$ to $\Hom(C,-)$ consists of functors $F$ such that $\coprod_{m\in\mathbb{Z}}F(X[m])$ is finite dimensional for all $X\in\T_c^b$. 
\end{Lem}

Note that the assumption in (2) is satisfied if $\T=\D[A]$ is the derived category of a non-positive dg algebra $A$ such that  $\coprod_{m\in\mathbb{Z}}H^m(A)$ is finite dimensional(see \cite[Lemma 5.2]{Goodbody2025}). 

\begin{Prop}\label{prop-SMC2silting}
Suppose that there is some $G'\in\T_c^b$ and $N>0$ such that $\T=\overline{\langle G'\rangle}_N$. For a simple minded collection $X_1,X_2,\cdots,X_r$ in $\T_c^b$, there are  objects $T_1,T_2,\cdots, T_r$ in $\T^c$ such that 
\[\Hom(T_j,X_i[m])\simeq \begin{cases}
	{}_{R_i}R_i, & i=j, m=0,\\
	0, & \mbox{else},
\end{cases}\]
and that $T=T_1\oplus T_2\oplus\cdots \oplus T_r$ is a silting object in $\T^c$. 
\end{Prop}

\begin{proof}
Let $I_1,I_2,\cdots,I_r$ be the objects we constructed in Lemma \ref{lem-smc2cosilting}. 
Then for each $i=1,2,\cdots,r$, $\coprod_{m\in\mathbb{Z}}D\Hom(X[m],I_i)$ is finite dimensional for all $X\in\T_c^b$, and the functor $D\Hom(-,I_i)$ is isomorphic to $\Hom(T_i, -)$ on $\T_c^b$ for some $T_i\in\T^c$ by Lemma \ref{lem-representability}, that is, there is a natural isomorphism 
\[D\Hom(-,I_i)|_{\T_c^b}\simeq\Hom(T_i,-)|_{\T_c^b}.\] 
The isomorhism can be extended to $\T_c^-$.  Since $T_i\in\T^c$, there exists a positive integer $m$ such that 
\[T_i\in\add(G[-m]*G[-m+1]*\cdots * G[m]).\]
By Lemma \ref{lem-smc2cosilting}, we can assume that $I_i\in\T_G^{\geq l}$. Choose $t$ sufficiently small such that $t<-m-1$ and $t<l$. Hence 
\[\Hom(T_i,\T_G^{\leq t})=0=\Hom(\T_G^{\leq t}, I_i)\]
Now for each $X\in\T_c^-$. Our assumption makes sure that the $t$-structure $(\T_G^{\leq 0}, \T_G^{\geq 0})$ restricts to a $t$-structure on $\T_c^-$, and there is a canonical triangle 
\[X^{<t}\lra X\lra X^{\geq t}\lra X^{<t}[1]\]
with $X^{\geq t}\in\T_c^b$. Applying $D\Hom(T_i, -)$ to this triangle gives a commutative diagram 
\[\xymatrix{
0=D\Hom(T_i, X^{<t}[1])\ar[r] & D\Hom(T_i, X^{\geq t})\ar[r] & D\Hom(T_i, X)\ar[r] & D\Hom(T_i, X^{<N})=0\\
0=\Hom(X^{<t}, I_i)\ar[r] & \Hom(X^{\geq t}, I_i)\ar[r]\ar[u]^{\simeq} &  \Hom(X, I_i)\ar[r]\ar[u]^{\simeq} &  \Hom(X^{<t}, I_i)=0
}\]
Thus we get an natural isomorphism $D\Hom(T_i, X)\simeq \Hom(X, I_i)$ for all $X\in\T_c^-$.

Next we prove that $T$ is presilting. Since $\Hom(T_j, X_i[m])=0$ for all $i,j$  and $m<0$, the space  $\Hom(T_j,Z_i^{(n)}[m])=0$
for all $n$ and $m\leq 0$, where $Z_i^{(n)}$ is the object in the proof of Lemma \ref{lem-smc2cosilting}. Applying $\Hom(T_j, -)$ to the triangle  
\[Z_i^{(n)}\lraf{\alpha_i^{(n)}} X_i^{(n)}\lraf{\beta_i^{(n)}} X_i^{(n+1)}\lra Z_i^{(n)}[1]\]
gives an isomorphism $\Hom(T_j, \beta_i^{(n)}[m])$ for all $m<0$. Hence
\[\Hom(T_j, I_i[m])\simeq \colim \Hom(T_j, X_i^{(n)}[m])=0\]
for all $m<0$. The natural isomorphism 
$\Hom(T_j[m],T_i)\simeq D\Hom(T_i, I_j[m])$
then implies that $T$ is presilting.

Finally, we show that $\thick(T)=\T^c$. Let $I=I_1\oplus I_2\oplus\cdots\oplus I_r$. By Theorem \ref{theo-testing} and the duality $\Hom(T,-)\simeq D\Hom(-,I)$, it suffices to prove that, for each $0\neq Y\in\T_c^-$, there is some $n\in\mathbb{Z}$ such that $\Hom(Y[n], I)\neq 0$.

We claim that there exists some $i=1,2,\cdots,r$ and $t\in\mathbb{Z}$ such that 
$\Hom(Y, X_i[t])\neq 0$. Suppose the contrary, that is, $\Hom(Y, X_i[t])=0$ for all $i$ and $t$. Then $\Hom(Y,U[t])=0$ for all $U\in\thick(\mathcal{X})=\T_c^b$ and all $t\in\mathbb{Z}$. Since $0\neq Y$, there exists some $t$ such that $\Hom(G, Y[t])\neq 0$. Consider the canonical  triangle 
\[Y^{<t}\lra Y\lraf{\pi} Y^{\geq t}\lra Y^{<t}[1]\]
Then $\pi$ is not zero. However our assumption implies that $Y^{\geq t}\in\T_c^b$, which is a contradiction.

Since $\Hom(Y[t], X_i)$ is zero for sufficiently large $t$, let $m$ be the largest integer such that $\Hom(Y[m], X_i)\neq 0$ for some $i$. Recall that, by construction 
\[Z_i^{(n)}, X_i^{(n)}, \hocolim X_i^{(n)}\]
are all in $\T_G^{\geq l}$ for all $i,n$. Since $Y[m]\in\T_c^-$, there is a triangle
\[C\lra Y[m]\lra W\lra C[1]\]
with $W\in\T_G^{\leq l-2}$, then $\Hom(W, \T_G^{\geq l-1})=0$. It follows that 
$\Hom(Y[m], \T_G^{\geq l})\simeq \Hom(C, \T_G^{\geq l})$. 
The exact seuqence 
\[0=\Hom(C, Z_i^{(n)})\lra \Hom(C, X_i^{(n)})\lraf{(-,\beta_i^{(n)})} \Hom(C, X_i^{(n+1)})\]
implies that $(-,\beta_i^{(n)})$ is injective for all $n$. 
Together with the fact that  $\Hom(C, X_i^{(0)})=\Hom(Y[m], X_i)\neq 0$. It follows that 
\[\Hom(Y[m], I_i)\simeq \Hom(C, \hocolim X_i^{(n)})\simeq \colim \Hom(C, X_i^{(n)})\neq 0.\]
This finishes the proof. 
\end{proof}

Now we define maps in the following diagram. 
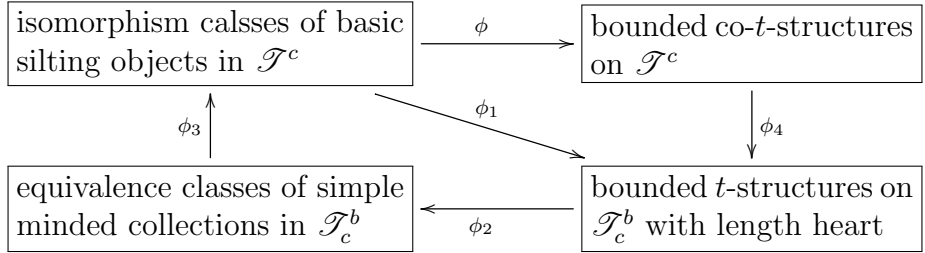
\[\xymatrix@C=20mm{
\fbox{\parbox{.3\textwidth}{isomorphism calsses of basic silting objects in $\T^c$}} \ar[r]^{\phi}\ar[rd]^{\phi_1} & \fbox{\parbox{.25\textwidth}{bounded co-$t$-structures on $\T^c$}}\ar[d]^{\phi_4}\\
\fbox{\parbox{.3\textwidth}{equivalence classes of simple minded collections in $\T_c^b$}}\ar[u]^{\phi_{3}} & \fbox{\parbox{.25\textwidth}{bounded $t$-structures on $\T_c^b$ with length heart}}\ar[l]^{\phi_{2}}  
}\]

We now introduce mutations and natural partial orders on each of the four classes of objects in the diagram above; these structures are deeply intertwined, reflecting the tight interplay between silting theory, $t$-structures, co-$t$-structures, and simple minded collections.

\medskip 
\textbf{Silting objects.} For two basic silting objects $M, N \in \T^c$, if $\Hom(M, N[m]) = 0$ for all integers $m > 0$, then we say that $M \geq N$. This relation defines a partial order on the isomorphism classes of basic silting objects in $\T^c$ \cite{Aihara2012aa}. To define mutation, start with a basic silting object $M = M_1 \oplus M_2 \oplus \cdots \oplus M_n$ decomposed into pairwise non-isomorphic indecomposable summands. For each index $i$, let $f_i: M_i \to E$ be a minimal left $\add\left(\bigoplus_{j \neq i} M_j\right)$-approximation. Complete $f_i$ to a distinguished triangle:
\[
M_i \lraf{f_i} E \to M_i^* \to M_i[1].
\]
The \emph{left mutation} of $M$ at $M_i$ is then $\mu_i^+(M) = M_i^* \oplus \bigoplus_{j \neq i} M_j$, which remains a basic silting object. Right mutations are defined dually via right minimal approximations \cite{Aihara2012aa}.

\textbf{Simple minded collections.} We work in $\T_c^b$. For two simple minded collections $\mathcal{X}, \mathcal{X}'$, set $\mathcal{X} \leq \mathcal{X}'$ if $\Hom(X, X'[m]) = 0$ for all $X \in \mathcal{X}$, $X' \in \mathcal{X}'$, and $m < 0$. This is a partial order on simple minded collections that arise as simple objects of bounded $t$-structures with length hearts \cite[Proposition 7.9]{Koenig2014aa}. For mutation, let $\mathcal{X} = \{X_1, X_2, \dots, X_r\}$ be a simple minded collection, and write $\mathcal{X}_i$ for the extension closure of $X_i$ in $\T_c^b$. By Hom-finiteness, each $X_j[-1]$ ($j \neq i$) admits a minimal left $\mathcal{X}_i$-approximation $g_{ji} \colon X_j[-1] \to X_{ji}$. Let $X_j^* = \cone(g_{ji})$ be the cone of this morphism. The left mutation of $\mathcal{X}$ at $i$ is:
\[
\mu_i^+(\mathcal{X}) = \{X_i[1]\} \cup \{X_j^* \mid j \neq i\}.
\]
For the bounded derived category of a finite-dimensional algebra, $\mu_i^+(\mathcal{X})$ is again a simple minded collection \cite{Koenig2014aa}; we will later verify this holds in $\T_c^b$ as well. Two simple minded collections are euqivalent if they are the same up to isomorphism and ordering. 

\textbf{Bounded t-structures.} Let $t_1, t_2$ be two bounded $t$-structures on $\T_c^b$. Define $t_1 \geq t_2$ if their non-positive aisles satisfy $t_1^{\leq 0} \supseteq t_2^{\leq 0}$. Now take a bounded $t$-structure $t$ whose heart is a length category, and let $\mathcal{X} = \{X_1, \dots, X_r\}$ be the simple objects of this heart (a simple minded collection by construction). The left mutation $\mu_i^+(t)$ is the unique bounded $t$-structure whose heart admits $\mu_i^+(\mathcal{X})$ as its simple objects—this is well-defined thanks to the simple minded collection mutation above.

\textbf{Co-t-structures.} Finally, consider co-$t$-structures on $\T^c$. For two co-$t$-structures $c, c'$, set $c \geq c'$ if and only if their non-positive co-aisles satisfy $c_{\leq 0} \supseteq c'_{\leq 0}$. Recall that a bounded co-$t$-structure $c$ has a co-heart with a basic additive generator $M = M_1 \oplus \cdots \oplus M_r$, which is a silting object. For each $i$, define $c'_{\leq 0}$ as the additive closure of the extension closure of $\{M_j[m], M_i[m+1] \mid j \neq i, m \geq 0\}$, and set $c'_{\geq 0} = {}^{\perp}(c'_{\leq 0}[1])$ (the left orthogonal of the shifted co-aisle). The pair $\mu_i^+(c) := (c'_{\geq 0}, c'_{\leq 0})$ is a bounded co-$t$-structure, called the left mutation of $c$ at $i$.

\medskip 
We now define the maps featured in the commutative diagram above, and verify their well-definedness and core properties in sequence.

First, the map $\phi$ sends a basic silting object $M$ to the bounded co-$t$-structure 
\[\big(\langle M\rangle^{[0,+\infty)}, \langle M\rangle^{(-\infty,0]}\big).\] It is a standard result that $\phi$ is a bijection, whose inverse $\phi^{-1}$ maps any bounded co-$t$-structure on $\T^c$ to the basic additive generator of its co-heart. 

Next, we introduce the map $\phi_1$, which takes basic silting objects in $\T^c$ to bounded $t$-structures on $\T_c^b$ with length hearts. Fix a basic silting object $M \in \T^c$. In our setup, $\T^c$ embeds into $\T_c^b$, and $\add(M)$ is contravariantly finite in $\T_c^b$. By Theorem \ref{theo-main-intro}, the pair
$$
\big( \T_M^{\leq 0} \cap \T_c^b, \T_M^{\geq 0} \cap \T_c^b \big)
$$
defines a bounded $t$-structure on $\T_c^b$, whose heart is equivalent to $\modcat{H^0(M)}$, and hence further equivalent to $\modcat{\End_{\T}(M)}$. Since $\End_{\T}(M)$ is a finite-dimensional algebra, its module category is a length category, so the heart of the $t$-structure above is also a length category. This completes the verification of the well-definedness of $\phi_1$.

The map $\phi_4$ sends a co-$t$-structure $(c_{\geq 0}, c_{\leq 0})$ to a bounded $t$-structure $(t^{\leq 0}, t^{\geq 0})$ on $\T_c^b$, where 
\[t^{\leq 0}=(c_{\geq 1})^{\perp}\cap\T_c^b.\]
If $(c_{\geq 0}, c_{\leq 0})=\phi(M)$ for some silting object $M$, then it follows by definition that 
\[\T_c^b\cap\T_{M}^{\leq 0}=(\langle M\rangle^{[1,+\infty]})^{\perp}\cap\T_c^b.\]
Hence $\phi_4\phi=\phi_1$.

We now turn to the map $\phi_2$, whose domain is the set of bounded $t$-structures on $\T_c^b$ with length hearts, and whose codomain is the set of simple minded collections in $\T_c^b$. For any such $t$-structure, $\phi_2$ sends it to the set of isomorphism classes of simple objects in its heart, which is a standard fact to form a simple minded collection in $\T_c^b$.

Finally, under the assumptions of Proposition \ref{prop-SMC2silting}, we define the map $\phi_3$ from the set of simple minded collections in $\T_c^b$ to the set of isomorphism classes of basic silting objects in $\T^c$. For any simple minded collection $\mathcal{X}$, $\phi_3$ maps $\mathcal{X}$ to the isomorphism class of the basic silting object $T$, whose existence and explicit construction are given in Proposition \ref{prop-SMC2silting}.

%% mutation of a simple-minded collection corresponds to mutation of silting.

\begin{Lem}\label{lem-bijection}
Keep the notations above and assume that there is some $G'\in\T_c^b$ and $N>0$ such that $\T=\overline{\langle G'\rangle}_N$. Then the following hold. 
\begin{enumerate}
	\item $\phi_2\phi_1$ is injective;
	\item $\phi_2\phi_1\phi_3$ is the identity map on the set of simple minded collections in $\T_c^b$;
	\item  $\phi_1\phi_3\phi_2$ is the identity map on the set of bounded $t$-structures on $\T_c^b$ with length heart.
	\item $ \phi_1,\phi_2,\phi_3$ all are bijections. 
\end{enumerate}
\end{Lem}
 
\begin{proof}
(1) By definition, $\phi_2\phi_1$ sends a silting object $M$ to the collection $\mathcal{X}$ of simple objects in the heart $\mathcal{H}_M=\T_c^b\cap\T_M^{\leq 0}\cap\T_M^{\geq 0}$ 
of the $t$-structure $(\T_c^b\cap\T_M^{\leq 0},\T_c^b\cap\T_M^{\geq 0})$ on $\T_c^b$.  Now suppose $N$ is another basic silting object in $\T^c$ such that $\mathcal{X}$ also is the collection of simple objects in $\mathcal{H}_N=\T_c^b\cap\T_N^{\leq 0}\cap\T_N^{\geq 0}$. Since the heart is generated by the simple objects by iterated triangle extension, we deduce that $\mathcal{H}_M=\mathcal{H}_N$. It follows that the two bounded $t$-structures coincide. Particularly $\T_c^b\cap\T_M^{\leq 0}=\T_c^b\cap\T_N^{\leq 0}$. Hence $\Hom(M,N[i])=0$ for all $i>0$. This means $M\geq N$ in the partially ordered set consisting of basic silting objects in $\T^c$. By symmetry, we also have $N\geq M$, and thus $M\simeq N$. 

(2) Starting from a simple minded collection $\mathcal{X}=\{X_1,X_2,\cdots,X_r\}$ in $\T_c^b$, Proposition \ref{prop-SMC2silting} constructs a basic silting object $T=T_1\oplus T_2\oplus\cdots\oplus T_r$ in $\T^c$ such that 
\[\Hom(T_j,X_i[m])\simeq \begin{cases}
	{}_{R_i}R_i, & i=j, m=0,\\
	0, & \mbox{else},
\end{cases}\quad\quad (\dagger)\]
where $R_i$ is the endomorphism algebra of $X_i$, $i=1,2,\cdots,r$. 
The image $\phi_2\phi_1\phi_3(\mathcal{X})$ is the collection of simple objects in the heart $\mathcal{H}_T$ of the bounded $t$-structure $(\T_c^b\cap\T_T^{\leq 0},\T_c^b\cap\T_T^{\geq 0})$ on $\T_c^b$. By Theorem \ref{theo-main-intro}, there is an equivalence 
\[\mathcal{H}_T\lra \modcat{\add(T)},\quad  X\mapsto\Hom(-,X).\]
Equivalently, $\mathcal{H}_T\lra \modcat{\End(T)}, X\mapsto \Hom(T,X)$ is an equivalence. The above $(\dagger)$ shows that $\Hom(T,X_i)$ is the simple $\End(T)$-module correponding to the indecomposable projective module $\Hom(T,T_i)$. The equivalence then tells us that $X_1,X_2,\cdots,X_r$ are precisely those simple objects in $\mathcal{H}_T$. This proves that $\phi_2\phi_1\phi_3(\mathcal{X})=\mathcal{X}$. 

(3) For each bounded $t$-structure ${\sf t}=(t^{\leq 0},t^{\geq 0})$ with length heart $\mathcal{H}$. The image of this $t$-structure under $\phi_2$ is the collection $\mathcal{X}$ of all simple objects in $\mathcal{H}$. Let $T$ be as constructed in Proposition \ref{prop-SMC2silting}. Then the image $\phi_1\phi_2(\mathcal{X})$ is the $t$-structure $(\T_c^b\cap\T_T^{\leq 0},\T_c^b\cap\T_T^{\geq 0})$ on $\T_c^b$. By (2), the collection of simple objects in $\mathcal{H}_T$, that is, $\phi_3\phi_1\phi_2(\mathcal{X})$, coincides with $\mathcal{X}$. It follows that $\mathcal{H}=\mathcal{H}_T$. Hence $(t^{\leq 0}, t^{\geq 0})$ is precisely the $t$-structure $(\T_c^b\cap\T_T^{\leq 0},\T_c^b\cap\T_T^{\geq 0})$. This shows that $\phi_1\phi_3\phi_2({\sf t})={\sf t}$.  

(4) The bijectivity follows from (1), (2) and (3) easily. Let $M=M_1\oplus M_2\oplus\cdots\oplus M_r$ be a basic silting object in $\T^c$ and let $\mathcal{X}=\{X_1,X_2,\cdots, X_r\}=\phi_3^{-1}(M)$ be the correponding simple minded collection. By definition, $\mathcal{X}$ consists precisely those simple objects in the heart $\mathcal{H}_M$ of the $t$-structure on $\T_c^b$ induced by $M$. Moreover, $\mathcal{H}_M$ is equivalent to $\modcat{\End(M)}$ via the functor sending $X$ to $\Hom(M,X)$. Identifying $\mathcal{H}_M$ with $\modcat{\End(M)}$, the proof of \cite[Lemma 7.8]{Koenig2014aa} still works. This implies that $\mu_i^+(\mathcal{X})$ is again a simple minded collection for each $i=1,2,\cdots,r$. 
\end{proof}

To prove that the bijections commute with mutations and preserve partial orders, we need the following lemma. 
\begin{Lem}\label{lem-fd-iso}
Let $\Lambda$ be a finite dimensional basic algebra such that  $\Lambda_{\Lambda}=P_1\oplus P_2\oplus\cdots\oplus P_n$ is a direct sum of indecomposable projecitve modues.  For each $i=1,2,\cdots,r$, let $P_i\lraf{f} E$ be a minimal left $\add(\bigoplus_{k\neq i}P_k)$-approximation of $P_i$. Suppose $j\neq i$ and $\Delta_j$ is the maximal quotient of $P_j$ such that the radical of $\Delta_j$ only has the top of $P_i$ as composition factors. Then the induced morphism $\Hom_{\Lambda}(E,\Delta_j)\lraf{(f,-)}\Hom_{\Lambda}(P_i,\Delta_j)$ is an isomorphism. 
\end{Lem}

\begin{proof}
We denote by $S_k$ the simple top of the indecomposable projective modue $P_k$ for each $k$. The module $\Delta_j$ can be interpreted as follows. Take a right $\add(\bigoplus_{k\neq i}P_k)$-approximation of $\rad P_j$, say $\rho: E'\lra \rad P_j$. One can easily check that $\Delta_j$ is just the cokernel of the composition of $\rho$ and the inclusion $\rad P_j\hookrightarrow P_j$, still denoted by $\rho$. Thus we have an exact sequence 
\[E'\lraf{\rho} P_j\lraf{\pi}\Delta_j\lra 0,\] 
where $\pi$ is the canonical epimorphism.

We first prove the 
 surjectivity of $(f,-)$.   For each morphism $g: P_i\lra \Delta_j$, we can form the following diagram. 
\[\xymatrix@!=10mm{
P_i\ar[r]^{f}\ar[rd]_(.3){g}\ar@{-->}[d]_{h} & E\ar@{-->}[ld]^(.3){\theta}\\
P_j\ar[r]_{\pi} & \Delta_j
}\]
The morphism $h$ is induced by $g$ such that $\pi h=g$. Since $f$ is a left $\add(\bigoplus_{k\neq i}P_i)$-approximation, the morphism $h$ factorizes through $f$, that is, there is a morphism $\theta: E\lra P_j$ such that $\theta f=h$. It follows that $g=\pi h=\pi\theta f=(f,-)(\pi\theta)$. Hence $(f,-)$ is surjective.  

Next we consider the injectivity. For this purpose, we rewrite $E$ as $E=P_j^m\oplus Q$, where $P_j^m$ is a direct sum of $m$ copies of $P_j$ and $Q$ has no direct summands isomorphic to $P_j$. The morphism $f$ then is rewritten as $f=[f_1,\cdots,f_m,f']^T: P_i\lra P_j^m\oplus Q$. Now let $g=[g_1,\cdots,g_m,g']: P_j^m\oplus Q\lra \Delta_j$ be a morphism such that $gf=0$. We need to prove that $g=0$. Now $g': Q\lra \Delta_j$ lifts to a morphism $g'':Q\lra P_j$ which further factorizes through $\rho: E'\lra P_j$. This implies that $g'=\pi g''=\pi \rho \gamma=0$ for some $\gamma: Q\lra E'$. If $m=0$, then clearly $g=0$. Now assume that $m>0$. Then $gf=g_1f_1+g_2f_2+\cdots+g_mf_m=0$.  By definition the multiplicity of $S_j$ as a composition factor of  $\Delta_j$ is $1$. Also $S_j$ is the top of $\Delta_j$. It follows that $R_j:=\End_{\Lambda}(\Delta_j)\simeq \End_{\Lambda}(S_j)$ is a division algebra and $\Hom_{\Lambda}(P_j,\Delta_j)$ is $1$-dimensonal over $R_j$. Hence there are $r_1,r_2,\cdots,r_m\in\End_{\Lambda}(\Delta_j)$ such that $g_k=r_k\pi$ for all $k=1,2,\cdots,m$. Thus $r_1\pi f_1+r_2\pi f_2+\cdots+r_m\pi f_m=0$. It remains to show that $\pi f_1,\pi f_2,\cdots, \pi f_m$ are $R_j$-linearly independent. If not, then some $\pi f_k$ can be written as a $R_j$-linear combination of the remaining morphisms. Without loss of generality, we assume that 
\[\pi f_1= r'_1\pi f_2+r'_3\pi f_2+\cdots+r'_{m}\pi f_{m}\]
for some  $r_2',r'_3,\cdots,r'_{m}$ in $R_j$. Each $r_k'$ lifts to an endomorphism $h_k\in \End_{\Lambda}(P_j)$ such that $\pi h_k=r'_k \pi$. Thus 
\[\pi (f_1-\sum_{k=2}^{m}h_kf_k)=\pi f_1-\sum_{k=2}^m \pi h_kf_k=\pi f_1-\sum_{k=2}^m r'_k \pi f_k=0.\]
It follows that $f_1-\sum_{k=2}^{m}h_kf_k$ factorizes through $\Ker\pi=\Img\rho$. There is a morphism $\theta: P_i\lra E'$ such that  $f_1-\sum_{k=2}^{m}h_kf_k=\rho\theta$. Since $f$ is left $\add(\bigoplus_{k\neq i}P_i)$-approximation, $\theta=u f$ for some $u: E\lra E'$. Now rewrite $u$ as $u=[u_1,u_2,\cdots,u_m,u']: P_j^m\oplus Q\lra E'$. Then 
\[f_1-\sum_{k=2}^{m}h_kf_k=\rho\theta=\rho u f=\rho u_1f_1+\rho u_2f_2+\cdots \rho u_mf_m+\rho u'f'.\]
Note that the morphism $\rho u_1$ lies in the radical of $\End_{\Lambda}(P_j)$ since $\Img \rho\subseteq \rad(P_j)$. It follows that $1-\rho u_1$ is invertible in $\End_{\Lambda}(P_j)$. Hence 
\[f_1=(1-\rho u_1)^{-1}\left(\rho u'f'+\sum_{k=2}^m (h_k+\rho u_k)f_k\right)\]
fanctorizes through the morphism $[f_2,f_3,\cdots, f_m, f']^T: P_i\lra P_j^{m-1}\oplus Q$. This contradicts to the assumption that $f$ is a minimal left $\add(\bigoplus_{k\neq i}P_i)$-approximation. 
\end{proof}

Now we are ready to give a proof of Theorem \ref{theo-correspondence}

\begin{proof}[Proof of Theorem \ref{theo-correspondence}]
	
The bijections follow from Lemma \ref{lem-bijection}. It remains to show that these bijections commute with mutations and preserve partial orders. 

That $\phi$ commutes with mutations and preserve partial orders was proved already in \cite[Proposition 7.11]{Koenig2014aa} and \cite[Proposition 2.14]{Aihara2012aa}. 

For each basic silting object $M$ in $\T^c$, suppose that $M=M_1\oplus M_2\oplus \cdots \oplus M_r$ is a direct sum of indecomposable objects. By definition, the left mutation $\mu_i^+(M)$ is $M^*=M_i^*\oplus\bigoplus_{k\neq i}M_k$, where $M_i^*$ is the cone of a minimal left $\add(\bigoplus_{k\neq i}M_k)$-approximation $f: M_i\lra E$ of $M_i$. The bijection $\phi_3^{-1}$ takes $M$ to the collection $\mathcal{X}=\{X_1,X_2,\cdots, X_r\}$ of simple objects in the heart $\mathcal{H}_M$ of the $t$-structure on $\T_c^b$ induced by $M$. We shall show that $\mathcal{X}^*=\mu_i^+(\mathcal{X})$ can be defined and $\phi_3^{-1}(M^*)=\mathcal{X}^*$.  By Theorem \ref{theo-main-intro}, the heart $\mathcal{H}_M$ is equivalent to module category of the finite dimensional algebra $\Lambda=\End(M)$ by sending $X$ to $\Hom(M,X)$. For each $j\neq i$, one can construct a diagram of  $\Lambda$-modules.
\[\xymatrix{
\Hom(M,E')\ar[r]^{\rho} & \Hom(M,M_j) \ar[r]^{\pi}\ar@{->>}[rd] & \Hom(M,X_j)\ar[r] & 0\\
&&\Hom(M,X_j^*)\ar[u]_{\pi'}\\
&&\Hom(M,X_{ij})\ar[u]
}\]
The morphism $\pi$ is the canonical map from the projective $\Lambda$-module $\Hom(M,M_j)$ to its top $\Hom(M,X_j)$, and $\rho$ is a right $\add(\Hom(M,\bigoplus_{k\neq i}M_k))$-approximation of $\Ker\pi=\rad \Hom(M,M_j)$. The cokernel of $\rho$ is $\Hom(M,X_j^*)$ for some $X_j^*\in\mathcal{H}_M$ and the kernel of the induced morphism $\pi'$ is $\Hom(M,X_{ij})$ for some $X_{ij}\in\mathcal{H}_M$. By construction $\Hom(M,X_{ij})=\rad\Hom(M,X_j^*)$ has no composition factors of the form $\Hom(M,X_k), k\neq i$. Accordingly $X_{ij}$ belongs to the extension closure $\mathcal{X}_i$ of $X_i$. Since $E'\in\add(\bigoplus_{k\neq i}M_k)$, we deduce that $\Ext_{\Lambda}^1(\Hom(M,X_j^*),\Hom(M,X_i))=0$. It follows that, for each $X\in\mathcal{X}_i$,  
\[\Hom(X_j^*[-1], X)\simeq \Hom(X_j^*, X[1])\simeq \Ext_{\mathcal{H}_M}^1(X_j^*,X)\simeq \Ext_{\Lambda}^1(\Hom(M,X_j^*),\Hom(M,X))=0.\]
Thus, in the triangle $\xi_j: X_j[-1]\lraf{g} X_{ij}\lra X_j^*\lra X_j$, the morphism $g$ is a left $\mathcal{X}_i$-approximation. Since $\Hom(M,X_j^*)$ and whence $X_j^*$ is indecomposable, $g$ must be a minimal left $\mathcal{X}_i$-approximation. This shows that $\mu_i^+(\mathcal{X})$ is defined and consists of $X_i[1]$ and those $X_j^*,j\neq i$. We shall show that these objects are precisely the simple objects in the heart $\mathcal{H}_{M^*}$. It follows from  $X_{ij}\in\mathcal{X}_i$ that $\Hom(M_k[m],X_{ij})=0$ for all $k\neq i$ and all integers $m$. Applying $\Hom(M_k[m],-), k\neq i$ to the triangle $\xi_j$ results in an isomorphism 
\[\Hom(M_k[m], X_j^*)\simeq \begin{cases}
	\End(X_j), & m=0, k=j;\\
	0,& \text{else.}  
\end{cases} \]
For each $X\in\mathcal{H}_M$, one has $\Hom(M[m],X)=0$ for all $m\neq 0$. Applying $\Hom(-,X)$ to the triangle $M_i\lraf{f} E\lra M_i^*\lra M_i[1]$, for each integer $m$, there is an exact sequence 
\[\Hom(E[m+1],X)\lra \Hom(M_i[m+1],X)\lra \Hom(M_i^*[m], X)\lra \Hom(E[m],X),\]
which implies that $\Hom(M_i^*[m],X)=0$ for all $m\neq 0, -1$. Hence $\Hom(M_i^*, X_j^*[m])=0$ for all $m\neq 0,1$. For the special case $X\in\mathcal{X}_i$, the Hom-space $\Hom(E[m],X)$ vanishes for all $m$, and consequently $\Hom(M_i^*,X[m])=0$ for all $m\neq 1$. For $X=X_i$, the above sequence gives 
\[\Hom(M_i^*, X_i[1])\simeq \Hom(M_i[1],X_i[1])\simeq \End(X_i).\] 
Taking $X=X_j^*$ in the above exact sequence, we get another exact sequence 
\[0\lra \Hom(M_i^*, X_j^*)\lra \Hom(E, X_j^*)\lraf{(f,-)} \Hom(M_i, X_j^*)\lra \Hom(M_i^*, X_j^*[1])\lra 0.\]
By construction, the $\Lambda$-module $\Hom(M, X_j^*)$ is $\Delta_j$ in Lemma \ref{lem-fd-iso}. The morphism 
\[\Hom(M, M_i)\lraf{(M,f)}\Hom(M,E)\]  
is a minimal left $\add\Hom(M,\bigoplus_{k\neq i}M_k)$-approximation. It follows from Lemma \ref{lem-fd-iso} that 
\[\Hom_{\Lambda}(\Hom(M,E),\Hom(M,X_j^*))\lra \Hom_{\Lambda}(\Hom(M,M_i),\Hom(M,X_j^*))\]
is an isomorphism, and consequently $(f,-)$ in the above sequence is an isomorphism. Hence both $\Hom(M_i^*,X_j^*)$ and $\Hom(M_i^*,X_j^*[1])$ vanish. Altogether, we have $\Hom(M_i^*,X_j^*[m])=0$ for all integers $m$. Finally, clearly one has $\Hom(M_j,X_i[m])=0$. Thus, we have proved that $\mu_i^+(\mathcal{X})=\{X_1[1], X_j^*,j\neq i\}$ is in the heart $\mathcal{H}_{M^*}$, and $\Hom(M^*, -)$ sends $X_i[1], X_j^*, j\neq i$ to simple $\End(M^*)$-modules. Since $\Hom(M^*, -)$ is an equivalence between $\mathcal{H}_{M^*}$ and $\modcat{\End(M^*)}$, the simple objects in $\mathcal{H}_{M^*}$ are $X_i[1], X_j^*, j\neq i$. Hence  $\phi_3^{-1}(M^*)=\mu_i^+(\mathcal{X})$. This shows that left mutation of simple minded collections in $\T_c^b$ can always be defined and $\phi_3^{-1}$ commutes with left mutations. The left mutation of a bounded  $t$-structures in $\T_c^b$ is defined via its collection of simple objects. It follows immediately that $\phi_1$ and $\phi_2$ commute with left mutations. Dually, one can prove that $\phi_1,\phi_2$ and $\phi_3$ commute with right mutations.

That the bijections $\phi_2$ and $\phi_4$ preserve partial orders was proved in \cite[Proposition 7.9, Theorem 7.13]{Koenig2014aa}. 
The desired result for  $\phi$  follows from \cite[Proposition 2.14]{Aihara2012aa}.  Since $\phi_1=\phi_4\phi$ and $\phi_3=(\phi_2\phi_1)^{-1}$, all the bijections preserve partial orders. 
\end{proof}

\section*{Acknowledgements}
The authors are partially supported by  Beijing Natural Science Foundation (1252011).

\bigskip 
\noindent 
Wei Hu, School of Mathematical Sciences, Beijing Normal University, Beijing 100875, China

\noindent
\medskip 
Email: {\tt huwei@bnu.edu.cn}

\noindent
\medskip 
Ziheng Liu, School of Mathematical Sciences, Beijing Normal University, Beijing 100875, China

\noindent
\medskip 
Email: {\tt  alglzh@mail.bnu.edu.cn}


\begin{thebibliography}{10}

\bibitem{Aihara2012aa}
T.~Aihara and O.~Iyama.
\newblock Silting mutation in triangulated categories.
\newblock {\em J. Lond. Math. Soc. (2)}, 85(3):633--668, 2012.

\bibitem{Tarrio2003}
L.~Alonso~Tarr\'io, A.~Jerem\'ias~L\'opez, and M.~J. Souto~Salorio.
\newblock Construction of {$t$}-structures and equivalences of derived categories.
\newblock {\em Trans. Amer. Math. Soc.}, 355(6):2523--2543, 2003.

\bibitem{Beilinson1982aa}
A.~A. Beilinson, J.~Bernstein, and P.~Deligne.
\newblock Faisceaux pervers.
\newblock In {\em Analysis and topology on singular spaces, {I} ({L}uminy, 1981)}, volume 100 of {\em Ast\'{e}risque}, pages 5--171. Soc. Math. France, Paris, 1982.

\bibitem{Beligiannis2007aa}
A.~Beligiannis and I.~Reiten.
\newblock Homological and homotopical aspects of torsion theories.
\newblock {\em Mem. Amer. Math. Soc.}, 883, 07 2007.

\bibitem{Biswas2024}
R.~Biswas, H.~Chen, K.~M. Rahul, C.~J. Parker, and J.~Zheng.
\newblock Bounded $t$-structures, finitistic dimensions, and singularity categories of triangulated categories, arXiv:2401.00130.

\bibitem{Bondarko2007aa}
M.~V. Bondarko.
\newblock Weight structures vs. $t$-structures; weight filtrations, spectral sequences, and complexes (for motives and in general).
\newblock {\em J. K-theory}, 6:387--504, 2007.

\bibitem{Fushimi2025}
R.~Fushimi.
\newblock The correspondence between silting objects and {$t$}-structures for non-positive dg algebras.
\newblock In {\em Proceedings of the 56th {S}ymposium on {R}ing {T}heory and {R}epresentation {T}heory}, pages 22--25. Symp. Ring Theory Represent. Theory Organ. Comm., Koganei, 2025.

\bibitem{Goodbody2025}
I.~Goodbody, T.~Raedschelders, and G.~Stevenson.
\newblock Approximable triangulated categories and reflexive dg-categories, {\em Appl. Cat. Str.} (2026) 34-19. 

\bibitem{Hoshino2002aa}
M.~Hoshino, Y.~Kato, and J.-I. Miyachi.
\newblock On $t$-structures and torsion theories induced by compact objects.
\newblock {\em J. Pure Appl. Algebra}, 167(1):15--35, 2002.

\bibitem{Keller2013}
B.~Keller and P.~Nicol\'as.
\newblock Weight structures and simple dg modules for positive dg algebras.
\newblock {\em Int. Math. Res. Not. IMRN}, (5):1028--1078, 2013.

\bibitem{Koenig2014aa}
S.~Koenig and D.~Yang.
\newblock Silting objects, simple-minded collections, {$t$}-structures and co-{$t$}-structures for finite-dimensional algebras.
\newblock {\em Doc. Math.}, 19:403--438, 2014.

\bibitem{Krause2020aa}
H.~Krause.
\newblock Completing perfect complexes.
\newblock {\em Math. Zeit.}, 296:1387--1427, 2020.

\bibitem{Krause2024}
H.~Krause.
\newblock Completions of triangulated categories.
\newblock In {\em Triangulated categories in representation theory and beyond---the {A}bel {S}ymposium 2022}, volume~17 of {\em Abel Symp.}, pages 169--193. Springer, Cham, [2024] \copyright 2024.

\bibitem{Marks2023}
F.~Marks and A.~Zvonareva.
\newblock Lifting and restricting t-structures.
\newblock {\em Bull. Lond. Math. Soc.}, 55(2):640--657, 2023.

\bibitem{Hernandez2013}
O.~Mendoza~Hern\'andez, E.~C. S\'aenz~Valadez, V.~Santiago~Vargas, and M.~J. Souto~Salorio.
\newblock Auslander-{B}uchweitz context and co-{$t$}-structures.
\newblock {\em Appl. Categ. Structures}, 21(5):417--440, 2013.

\bibitem{Neeman2018TheC}
A.~Neeman.
\newblock The categories $\mathcal{T}^c$ and $\mathcal{T}^b_c$ determine each other.
\newblock arXiv:1806.06471v1.

\bibitem{Neeman2021tstr}
A.~Neeman.
\newblock The {$t$}-structures generated by objects.
\newblock {\em Trans. Amer. Math. Soc.}, 374(11):8161--8175, 2021.

\bibitem{Neeman2024}
A.~Neeman.
\newblock Bounded {$t$}-structures on the category of perfect complexes.
\newblock {\em Acta Math.}, 233(2):239--284, 2024.

\bibitem{Neeman2026aa}
A.~Neeman.
\newblock Triangulated categories with a single compact generator and a brown representability theorem.
\newblock {\em Invent. Math.}, 2026.

\bibitem{Pauks2008aa}
D.~Pauksztello.
\newblock Compact corigid objects in triangulated categories and co-$t$-structures.
\newblock {\em Central Eur. J. Math.}, 6:25--42, 2008.

\bibitem{Rahul2025}
K.~M. Rahul.
\newblock Representability theorems via metric techniques, 	arXiv:2504.11768.

\bibitem{Rickard2002aa}
J.~Rickard.
\newblock Equivalences of derived categories for symmetric algebras.
\newblock {\em J. Algebra}, 257(2):460--481, 2002.

\bibitem{Saorin2022}
M.~Saor\'in and A.~Zvonareva.
\newblock Lifting of recollements and gluing of partial silting sets.
\newblock {\em Proc. Roy. Soc. Edinburgh Sect. A}, 152(1):209--257, 2022.

\bibitem{Su2019}
H.~Su and D.~Yang.
\newblock From simple-minded collections to silting objects via {K}oszul duality.
\newblock {\em Algebr. Represent. Theory}, 22(1):219--238, 2019.

\end{thebibliography}
\end{document}